\documentclass[11pt]{preprint}

\usepackage{difftrees} %Added%
\usepackage[full]{textcomp}
\usepackage[osf]{newtxtext} 
\usepackage[cal=boondoxo]{mathalfa}
\usepackage{colortbl}

\usepackage{tikz-cd} %Added%
\usetikzlibrary{cd} %Added%
\usepackage[all,cmtip]{xy}
\usepackage{comment}

\usepackage{amssymb}
\usepackage{mathtools}
\usepackage{hyperref}
\usepackage{breakurl}
\usepackage{mhenvs}
\usepackage{mhequ} 
\usepackage{mhsymb}
\usepackage{booktabs}
\usepackage{tikz}
\usepackage{tcolorbox}
\usepackage{mathrsfs}
\usepackage[utf8]{inputenc}
\usepackage{longtable}
\usepackage{wrapfig}
\usepackage{rotating} %Added%
\usepackage{subcaption}
\usepackage{mathrsfs}
\usepackage{epsfig}
\usepackage{microtype}
\usepackage{comment}
\usepackage{wasysym}
\usepackage{centernot}
\usepackage{enumitem}
\usepackage{bm}
\usepackage{stackrel}
\usepackage{graphicx}
\usepackage{axodraw}
\usepackage{xspace}
\usepackage{subcaption} %Added%
\usepackage{epsfig} %Added%
\usepackage{axodraw} %Added%
\usepackage{xspace} %Added%
\usepackage[toc,page]{appendix} %Added%
\usepackage[all,cmtip]{xy} %Added%
\usepackage{relsize} %Added%
\usepackage{shuffle} %Added%
\makeatletter
\newcommand{\globalcolor}[1]{%
  \color{#1}\global\let\default@color\current@color
}
\makeatother

\usepackage{forest}
\forestset{
	decor/.style = {},  % 关键修改：清空decor样式
	root/.style = {minimum size = 0.1ex},
	decorated/.style = {
		for tree = {
			circle, fill, inner sep = 0.3ex, minimum size = 0.4ex,
			edge={
				thick,
				shorten <=2pt,
				shorten >=2pt
			},
			grow' = north, l = 0, l sep = 1.0ex,
			s sep = 0.7em, fit = tight, parent anchor = center,
			child anchor = center,
			delay = {decor/.option = content, content =}
		}
	},
	default preamble = {decorated, root},
}

\usetikzlibrary{calc}
\usetikzlibrary{decorations}
\usetikzlibrary{positioning}
\usetikzlibrary{shapes}
\usetikzlibrary{external}

\definecolor{blush}{rgb}{0.87, 0.36, 0.51}
	\definecolor{brightcerulean}{rgb}{0.11, 0.67, 0.84}
	\definecolor{greenryb}{rgb}{0.4, 0.69, 0.2}

\newif\ifdark
\darkfalse

\ifdark
\definecolor{darkred}{rgb}{0.9,0.2,0.2}
\definecolor{darkblue}{rgb}{0.7,0.3,1}
\definecolor{darkgreen}{rgb}{0.1,0.9,0.1}
\definecolor{franck}{rgb}{0,0.8,1}
\definecolor{pagebackground}{rgb}{.15,.21,.18}
\definecolor{pageforeground}{rgb}{.84,.84,.85}
\pagecolor{pagebackground}
\AtBeginDocument{\globalcolor{pageforeground}}
\definecolor{symbols}{rgb}{0,0.7,1}
\colorlet{connection}{red!80!black}
\colorlet{boxcolor}{blue!50}

\else

\definecolor{darkred}{rgb}{0.7,0.1,0.1}
\definecolor{darkblue}{rgb}{0.4,0.1,0.8}
\definecolor{darkgreen}{rgb}{0.1,0.7,0.1}
\definecolor{franck}{rgb}{0,0,1}
\definecolor{pagebackground}{rgb}{1,1,1}
\definecolor{pageforeground}{rgb}{0,0,0}
\colorlet{symbols}{blue!90!black}
\colorlet{connection}{red!30!black}
\colorlet{boxcolor}{blue!50!black}

\fi

\def\slash{\leavevmode\unskip\kern0.18em/\penalty\exhyphenpenalty\kern0.18em}
\def\dash{\leavevmode\unskip\kern0.18em--\penalty\exhyphenpenalty\kern0.18em}

\DeclareMathAlphabet{\mathbbm}{U}{bbm}{m}{n}

\DeclareFontFamily{U}{BOONDOX-calo}{\skewchar\font=45 }
\DeclareFontShape{U}{BOONDOX-calo}{m}{n}{
  <-> s*[1.05] BOONDOX-r-calo}{}
\DeclareFontShape{U}{BOONDOX-calo}{b}{n}{
  <-> s*[1.05] BOONDOX-b-calo}{}
\DeclareMathAlphabet{\mcb}{U}{BOONDOX-calo}{m}{n}
\SetMathAlphabet{\mcb}{bold}{U}{BOONDOX-calo}{b}{n}
\setlist{noitemsep,topsep=4pt,leftmargin=1.5em}

\DeclareMathAlphabet{\mathbbm}{U}{bbm}{m}{n}

\DeclareMathAlphabet{\mcb}{U}{BOONDOX-calo}{m}{n}
\SetMathAlphabet{\mcb}{bold}{U}{BOONDOX-calo}{b}{n}
\DeclareFontFamily{U}{mathx}{\hyphenchar\font45}
\DeclareFontShape{U}{mathx}{m}{n}{
      <5> <6> <7> <8> <9> <10>
      <10.95> <12> <14.4> <17.28> <20.74> <24.88>
      mathx10
      }{}
\DeclareSymbolFont{mathx}{U}{mathx}{m}{n}
\DeclareMathSymbol{\bigtimes}{1}{mathx}{"91}

\providecommand{\figures}{false}
{ \ifthenelse{\equal{\figures}{false}} {#1}{\[ {\rm Figure \ missing !} \]} }{}
\def\id{\mathrm{id}}

\def\CH{\mathcal{H}}

\def\CC{\mathcal{C}}

\def\CB{\mathcal{B}}
\def\CM{\mathcal{M}}
\def\CT{\mathcal{T}}
\def\Cf{\mathcal{f}}

\def\ST{\mathscr{T}}

\tikzstyle{tinydots}=[dash pattern=on \pgflinewidth off \pgflinewidth]
\tikzstyle{superdense}=[dash pattern=on 4pt off 1pt]

\newcommand{\beq}{\begin{equation}}
\newcommand{\eeq}{\end{equation}}

\usepackage{empheq}

\def\${|\!|\!|}

\def\bstar{\mathbin{\bar{\star}}}

\newcounter{theorem}

\newtheorem{ex}[theorem]{Example}

\newenvironment{DIFnomarkup}{}{} % see man latexdiff

\def\BCK{\textnormal{\tiny \textsc{BCK}}}

\newtheorem{assumption}{Assumption}
 
\theorembodyfont{\rmfamily}

\newcommand{\rrightarrow}{{\to\hskip -4.9mm\raise 1pt\hbox{$\to$}}}

\newfont{\indic}{bbmss12}

\def\Nabla_#1{\nabla_{\!#1}}

\makeatletter
\pgfdeclareshape{crosscircle}
{
  \inheritsavedanchors[from=circle] % this is nearly a circle
  \inheritanchorborder[from=circle]
  \inheritanchor[from=circle]{north}
  \inheritanchor[from=circle]{north west}
  \inheritanchor[from=circle]{north east}
  \inheritanchor[from=circle]{center}
  \inheritanchor[from=circle]{west}
  \inheritanchor[from=circle]{east}
  \inheritanchor[from=circle]{mid}
  \inheritanchor[from=circle]{mid west}
  \inheritanchor[from=circle]{mid east}
  \inheritanchor[from=circle]{base}
  \inheritanchor[from=circle]{base west}
  \inheritanchor[from=circle]{base east}
  \inheritanchor[from=circle]{south}
  \inheritanchor[from=circle]{south west}
  \inheritanchor[from=circle]{south east}
  \inheritbackgroundpath[from=circle]
  \foregroundpath{
    \centerpoint%
    \pgf@xc=\pgf@x%
    \pgf@yc=\pgf@y%
    \pgfutil@tempdima=\radius%
    \pgfmathsetlength{\pgf@xb}{\pgfkeysvalueof{/pgf/outer xsep}}%  
    \pgfmathsetlength{\pgf@yb}{\pgfkeysvalueof{/pgf/outer ysep}}%  
    \ifdim\pgf@xb<\pgf@yb%
      \advance\pgfutil@tempdima by-\pgf@yb%
    \else%
      \advance\pgfutil@tempdima by-\pgf@xb%
    \fi%
    \pgfpathmoveto{\pgfpointadd{\pgfqpoint{\pgf@xc}{\pgf@yc}}{\pgfqpoint{-0.707107\pgfutil@tempdima}{-0.707107\pgfutil@tempdima}}}
    \pgfpathlineto{\pgfpointadd{\pgfqpoint{\pgf@xc}{\pgf@yc}}{\pgfqpoint{0.707107\pgfutil@tempdima}{0.707107\pgfutil@tempdima}}}
    \pgfpathmoveto{\pgfpointadd{\pgfqpoint{\pgf@xc}{\pgf@yc}}{\pgfqpoint{-0.707107\pgfutil@tempdima}{0.707107\pgfutil@tempdima}}}
    \pgfpathlineto{\pgfpointadd{\pgfqpoint{\pgf@xc}{\pgf@yc}}{\pgfqpoint{0.707107\pgfutil@tempdima}{-0.707107\pgfutil@tempdima}}}
  }
}
\makeatother

\def\decorate#1#2{
        \ifnum#2>0
    		\foreach \count in {1,...,#2}{
	       	let
				\p1 = (sourcenode.center),
                \p2 = (sourcenode.east),
				\n1 = {\x2-\x1},
				\n2 = {1mm},
				\n3 = {(1.3+0.6*(\count-1))*\n1},
				\n4 = {0.7*\n1}
			in 
        		node[rectangle,fill=symbols,rotate=30,inner sep=0pt,minimum width=0.2*\n2,minimum height=\n2] at ($(sourcenode.center) + (\n3,\n4)$) {}
				}
		\fi
        \ifnum#1>0
    		\foreach \count in {1,...,#1}{
	       	let
				\p1 = (sourcenode.center),
                \p2 = (sourcenode.east),
				\n1 = {\x2-\x1},
				\n2 = {1mm},
				\n3 = {(1.3+0.6*(\count-1))*\n1},
				\n4 = {0.7*\n1}
			in 
        		node[rectangle,fill=symbols,rotate=-30,inner sep=0pt,minimum width=0.2*\n2,minimum height=\n2] at ($(sourcenode.center) + (-\n3,\n4)$) {}
				}
		\fi
}

\tikzset{
    dectriangle/.style 2 args={
        triangle,
        alias=sourcenode,
        append after command={\decorate{#1}{#2}}
    },
    dectriangle/.default={0}{0},
}

\tikzset{
	cross/.style={path picture={ 
  		\draw[symbols]
			(path picture bounding box.south east) -- (path picture bounding box.north west) (path picture bounding box.south west) -- (path picture bounding box.north east);
		}},
root/.style={circle,fill=green!50!black,inner sep=0pt, minimum size=1.2mm},
        dot/.style={circle,fill=pageforeground,inner sep=0pt, minimum size=1mm},
        dotred/.style={circle,fill=pageforeground!50!pagebackground,inner sep=0pt, minimum size=2mm},
        var/.style={circle,fill=pageforeground!10!pagebackground,draw=pageforeground,inner sep=0pt, minimum size=3mm},
        var2/.style={circle,fill=darkgreen,draw=pageforeground,inner sep=0pt, minimum size=3mm},
        var3/.style={circle,fill=brown,draw=pageforeground,inner sep=0pt, minimum size=3mm},
        var4/.style={circle,fill=purple,draw=pageforeground,inner sep=0pt, minimum size=3mm},
        var5/.style={circle,fill=blue,draw=pageforeground,inner sep=0pt, minimum size=3mm},
        var6/.style={circle,fill=darkblue,draw=pageforeground,inner sep=0pt, minimum size=3mm},
        kernel/.style={semithick,draw=green,shorten >=2pt,shorten <=2pt},
        kernels/.style={snake=zigzag,shorten >=2pt,shorten <=2pt,segment amplitude=1pt,segment length=4pt,line before snake=2pt,line after snake=5pt,},
        rho/.style={densely dashed,semithick,shorten >=2pt,shorten <=2pt},
           testfcn/.style={dotted,semithick,shorten >=2pt,shorten <=2pt},
        renorm/.style={shape=circle,fill=pagebackground,inner sep=1pt},
        labl/.style={shape=rectangle,fill=pagebackground,inner sep=1pt},
        xic/.style={very thin,circle,draw=symbols,fill=symbols,inner sep=0pt,minimum size=1.2mm},
        g/.style={very thin,rectangle,draw=symbols,fill=symbols!10!pagebackground,inner sep=0pt,minimum width=2.5mm,minimum height=1.2mm},
        xi/.style={very thin,circle,draw=symbols,fill=symbols!10!pagebackground,inner sep=0pt,minimum size=1.2mm},
	xies/.style={very thin,rectangle,fill=green!50!black!25,draw=symbols,inner sep=0pt,minimum size=1.1mm},
	xiesf/.style={very thin,rectangle,fill=green!50!black,draw=symbols,inner sep=0pt,minimum size=1.1mm},
        xix/.style={very thin,crosscircle,fill=symbols!10!pagebackground,draw=symbols,inner sep=0pt,minimum size=1.2mm},
        X/.style={very thin,cross,rectangle,fill=pagebackground,draw=symbols,inner sep=0pt,minimum size=1.2mm},
	xib/.style={thin,circle,fill=symbols!10!pagebackground,draw=symbols,inner sep=0pt,minimum size=1.6mm},
	xie/.style={thin,circle,fill=green!50!black,draw=symbols,inner sep=0pt,minimum size=1.6mm},
	xid/.style={thin,circle,fill=symbols,draw=symbols,inner sep=0pt,minimum size=1.6mm},
	xibx/.style={thin,crosscircle,fill=symbols!10!pagebackground,draw=symbols,inner sep=0pt,minimum size=1.6mm},
	kernels2/.style={very thick,draw=connection,segment length=12pt},
	keps/.style={thin,draw=symbols,->},
	kepspr/.style={thick,draw=connection,->},
	krho/.style={thin,draw=symbols,superdense,->},
	krhopr/.style={thick,draw=connection,superdense},
	triangle/.style = { regular polygon, regular polygon sides=3},
	not/.style={thin,circle,draw=connection,fill=connection,inner sep=0pt,minimum size=0.5mm},
	diff/.style = {very thin,draw=symbols,triangle,fill=red!50!black,inner sep=0pt,minimum size=1.6mm},
	diff1/.style = {very thin,dectriangle={1}{0},fill=red!50!black,draw=symbols,inner sep=0pt,minimum size=1.6mm},
	diff2/.style = {very thin,dectriangle={1}{1},fill=red!50!black,draw=symbols,inner sep=0pt,minimum size=1.6mm},
		diffmini/.style = {very thin,rectangle,fill=black,draw=black,inner sep=0pt,minimum size=0.75mm},
	 kernelsmod/.style={very thick,draw=connection,segment length=12pt},
	 rec/.style = {very thin,rectangle,fill=black,draw=black,inner sep=0pt,minimum size=2mm},
	cerc/.style={very thin,circle,draw=black,fill=symbols,inner sep=0pt,minimum size=2mm},
	stars/.style={very thin,star,star points=6,star point ratio=0.5, draw=black,fill=red,inner sep=0pt,minimum size=0.7mm},
	>=stealth,
        }
        \tikzset{
root/.style={circle,fill=black!50,inner sep=0pt, minimum size=3mm},
        circ/.style={circle,fill=white,draw=black,very thin,inner sep=.5pt, minimum size=1.2mm},
        round1/.style={fill=white,outer sep = 0,inner sep=2pt,rounded corners=1mm,draw,text=black,thin,minimum size=1.2mm},
          circ1/.style={circle,fill=red!10,draw=red,very thin,inner sep=.5pt, minimum size=1.2mm},
        rect/.style={fill=white,outer sep = 0,inner sep=2pt,rectangle,draw,text=black,thin,minimum size=1.2mm},
        rect1/.style={fill=white,outer sep = 0,inner sep=2pt,rectangle,draw,text=black,thin,minimum size=1.2mm},
        round2/.style={fill=red!10,outer sep = 0,inner sep=2pt,rounded corners=1mm,draw,text=black,thin,minimum size=1.2mm},
       round3/.style={fill=blue!10,outer sep = 0,inner sep=2pt,rounded corners=1mm,draw,text=black,thin,minimum size=1.2mm}, 
        rect2/.style={fill=black!10,outer sep = 0,inner sep=2pt,rectangle,draw,text=black,thin,minimum size=1.2mm},
        dot/.style={circle,fill=black,inner sep=0pt, minimum size=1.2mm},
        dotred/.style={circle,fill=black!50,inner sep=0pt, minimum size=2mm},
        var/.style={circle,fill=black!10,draw=black,inner sep=0pt, minimum size=3mm},
        kernel/.style={semithick,draw=darkgreen},
         diag/.style={thin,shorten >=4pt,shorten <=4pt},
        kernel1/.style={thick},
        kernels/.style={snake=zigzag,shorten >=2pt,shorten <=2pt,segment amplitude=1pt,segment length=4pt,line before snake=2pt,line after snake=5pt},
		kernels1/.style={snake=zigzag,segment amplitude=0.5pt,segment length=2pt},
		rho1/.style={densely dotted,semithick},
        rho/.style={densely dashed,semithick,shorten >=2pt,shorten <=2pt},
           testfcn/.style={dotted,semithick,shorten >=2pt,shorten <=2pt},
           visible/.style={draw, circle, fill, inner sep=0.25ex},
        renorm/.style={shape=circle,fill=white,inner sep=1pt},
        labl/.style={shape=rectangle,fill=white,inner sep=1pt},
        xic/.style={very thin,circle,fill=symbols,draw=black,inner sep=0pt,minimum size=1.2mm},
        xi/.style={very thin,circle,fill=blue!10,draw=black,inner sep=0pt,minimum size=1.2mm},
	xib/.style={very thin,circle,fill=blue!10,draw=black,inner sep=0pt,minimum size=1.6mm},
	xie/.style={very thin,circle,fill=green!50!black,draw=black,inner sep=0pt,minimum size=1mm},
	xid/.style={very thin,circle,fill=symbols,draw=black,inner sep=0pt,minimum size=1.6mm},
	edgetype/.style={very thin,circle,draw=black,inner sep=0pt,minimum size=5mm},
	nodetype/.style={very thick,circle,draw=black,inner sep=0pt,minimum size=5mm},
	kernels2/.style={very thick,draw=connection,segment length=12pt},
clean/.style={thin,circle,fill=black,inner sep=0pt,minimum size=1mm},	not/.style={thin,circle,fill=symbols,draw=connection,fill=connection,inner sep=0pt,minimum size=0.8mm},
	>=stealth,
        }

\makeatletter
\def\DeclareSymbol#1#2#3{%
	\expandafter\gdef\csname MH@symb@#1\endcsname{\tikzsetnextfilename{symbol#1}%
	\tikz[baseline=#2,scale=0.15,draw=symbols,line join=round]{#3}}%
	\expandafter\gdef\csname MH@symb@#1s\endcsname{\scalebox{0.75}{\tikzsetnextfilename{symbol#1}%
	\tikz[baseline=#2,scale=0.15,draw=symbols,line join=round]{#3}}}%
	\expandafter\gdef\csname MH@symb@#1ss\endcsname{\scalebox{0.65}{\tikzsetnextfilename{symbol#1}%
	\tikz[baseline=#2,scale=0.15,draw=symbols,line join=round]{#3}}}%
	}
\def\<#1>{\ifthenelse{\boolean{mmode}}{\mathchoice{\csname MH@symb@#1\endcsname}{\csname MH@symb@#1\endcsname}{\csname MH@symb@#1s\endcsname}{\csname MH@symb@#1ss\endcsname}}{\csname MH@symb@#1\endcsname}}
\makeatother

\DeclareMathAlphabet{\mathpzc}{OT1}{pzc}{m}{it}

\def\eqref#1{(\ref{#1})}

\def\rd{\mathrm{d}}

\makeatletter % Stolen from the internet to make a fat \cdot which isn't as fat as a \bullet
\newcommand*{\bigcdot}{}% Check if undefined
\DeclareRobustCommand*{\bigcdot}{%
  \mathbin{\mathpalette\bigcdot@{}}%
}
\newcommand*{\bigcdot@scalefactor}{.5}
\newcommand*{\bigcdot@widthfactor}{1.15}
\newcommand*{\bigcdot@}[2]{%
  \sbox0{$#1\vcenter{}$}% math axis
  \sbox2{$#1\cdot\m@th$}%
  \hbox to \bigcdot@widthfactor\wd2{%
    \hfil
    \raise\ht0\hbox{%
      \scalebox{\bigcdot@scalefactor}{%
        \lower\ht0\hbox{$#1\bullet\m@th$}%
      }%
    }%
    \hfil
  }%
}
\makeatother

\tcbset
{colframe=boxcolor,colback=symbols!7!pagebackground,coltext=pageforeground,
fonttitle=\bfseries,nobeforeafter,center title,size=fbox,boxsep=1.5pt,
top=0mm,bottom=0mm,boxsep=0mm,tcbox raise base}

\def\two{{\<generic>\kern0.05em\<genericb>}}
\def\twoI{{\<Ito>\kern0.05em\<Itob>}}

\def\mail#1{\burlalt{#1}{mailto:#1}}

\usepackage{thmtools} %Added%

\declaretheorem[style=definition]{example}

\begin{document}

\title{Banach fixed point and flow approach for rough analysis}
\author{\makebox[\textwidth][c]{Yvain Bruned$^1$, Yingtong Hou$^1$, Paul Laubie$^1$, Zhicheng Zhu$^2$}}
\institute{ 
 Universite de Lorraine, CNRS, IECL, F-54000 Nancy, France
 \and School of Mathematics and Statistics, Lanzhou University, Lanzhou, 730000, China
  \\
Email:\ \begin{minipage}[t]{\linewidth}
\mail{yvain.bruned@univ-lorraine.fr}
\\
\mail{yingtong.hou@univ-lorraine.fr}
\\
\mail{paul.laubie@univ-lorraine.fr}
\\
\mail{Zhicheng.Zhu_math@outlook.com}.
\end{minipage}}

\maketitle

\begin{abstract}
	In this paper, we show that the main algebraic assumption required to perform a fixed point argument for rough differential equations implies the algebraic assumption for the Bailleul flow approach. This assumption requires that the rough path associated with the equation is given by a Hopf algebra whose coproduct admits a cocycle and has a tree-like basis.
	We show that the Hopf algebra of multi-indices does not satisfy the cocycle condition. This is a rigorous result on the impossibility, observed in practice, of performing a fixed point argument for multi-indices rough paths and multi-indices in Regularity Structures.
\end{abstract}

\setcounter{tocdepth}{1}
\tableofcontents

\section{Introduction}
Decorated trees have become the main combinatorial set for solving a large class of singular stochastic partial differential equations (SPDEs) via the theory of Regularity Structures in \cite{reg,BHZ,CH16,BCCH}. The ansatz for the solution, a local expansion in terms of recentered stochastic iterated integrals follows the formalism given by B-series (see \cite{BCCH,BR23}). The B-series provides an efficient parametrisation.
These decorated trees come with two Hopf algebras in cointeraction: One for the recentering of integrals and the other for their renormalisation. The Hopf algebra for the recentering can be viewed as an extension of the so-called Butcher-Connes-Kreimer coproduct used for composing B-series in numerical analysis \cite{Butcher72} and encoding the nested subdivergences in QFT \cite{CK1}. A few years later, another combinatorial set called multi-indices emerged which is very efficient for describing scalar-valued equations such as quasi-linear SPDEs (see \cite{OSSW}) in the context of Regularity Structures. Its recentering Hopf  algebra is given in \cite{LOT} and one obtains a recursive proof for the stochastic estimates in \cite{LOTT} ; the proof for the stochastic estimates was first performed in \cite{CH16} with a non-recursive proof.
One of the drawbacks of the multi-indices approach for a while was the lack of a solution theory which is given in the context of decorated trees via Banach fixed point theorem.  
Such a path does not seem possible for multi-indices. A solution theory was uncovered in \cite{BBH26} for multi-indices rough paths introduced in \cite{Li23}. The main idea is to use the flow approach from Bailleul \cite{B15} that avoids the use of a fixed point with the idea of iterating a numerical scheme of the solution. This idea also appears in Davie's solutions for rough differential equations \cite{Davie} and the $\log$-ODE approach from \cite{CG96}. These approaches differ from the fixed point approach proposed in \cite{lyons1998,Gub06,Gubinelli2004} (see also the book \cite{FrizHai} for an introduction on rough differential equations).

In this work, we consider a rough differential equation (RDE) given by
\begin{equs}\label{RDE}
	\rd Y_t = f(Y_t) \rd X_{t} = \sum_{\alpha \in [d]} f_{\alpha}(Y_t) dX^{\alpha}_t
\end{equs}
where $ \alpha \in [d] = \{1,...,d \} $ ($X$ is $d$-dimensional), $Y$ is $m$-dimensional, and $t \in [0,T]$. The paths $X^\alpha$ for $ \alpha \in \{ 1,...,d \}$ are $\gamma$-Hölder for $\gamma \in (0,1)$. We assume that $X$ can be lifted to a rough path $ \mathbf{X}_{s,t} $ described by a graded Hopf algebra $ (\CH,\mu, \Delta) $, meaning that one has $ \mathbf{X}_{s,t} \in \mathcal{H}^{*}  $. We consider the dual of the previous Hopf algebra denoted by  $ (\CH^*, \star,\Delta_{\mu}) $ and we suppose given a pairing $ \langle \cdot, \cdot \rangle_{\CH} $ on $\CH$. This Hopf algebra gives Chen's relation
\begin{equation*}
	 \mathbf{X}_{s,t} =  \mathbf{X}_{s,u}  \star \mathbf{X}_{u,t}
\end{equation*}
with $\star$ the graded dual of $\Delta$. 
One well-know example is the Butcher-Connes-Kreimer Hopf algebra $ (\CH_{\BCK}, \odot, \Delta_{\BCK}) $ for branched rough paths \cite{Gub06}, where $ \odot $ is the forest product and $ \Delta_{\BCK} $ is the Butcher-Connes-Kreimer coproduct. The dual Hopf algebra is denoted by $(\CH_{\BCK}^*, \star_{\BCK}, \Delta_{\shuffle})$ where $\star_{\BCK}$ is the Grossman-Larson coproduct and $\Delta_{\shuffle}$ is the unshuffle coproduct. The inner product for $\CH_{\BCK}$ is denoted by $\langle \cdot, \cdot \rangle$. We denote $\Upsilon_f : \CH^*_{\BCK} \rightarrow \mathcal{C}^{\infty}$, the usual elementary differential. 
Then, one can apply the flow approach from \cite{B15} if the Hopf algebra satisfies some algebraic assumptions needed for getting elementary differentials and
 two key identities given in \cite[Definition 1.1]{KL23} that we recall below
\begin{assumption}[Algebraic flow condition] \label{flow_condition}
	We suppose that either 
	\begin{itemize}
		\item There exists a morphism $ \Lambda^* : \CH^* \rightarrow \CH_{\BCK}^* $.
	Then, one can define elementary differentials  as
	\begin{equs} \label{first_def_elementary}
		\bar{\Upsilon}_f[h] = \Upsilon_f[ \Lambda^*(h) ].
	\end{equs}
	\item There exists a morphism $ \Phi : \CH^*_{\BCK} \rightarrow \CH^* $ and elementary differentials $ \bar{\Upsilon}_f $ such that for every $\tau \in \CT$
	\begin{equs} \label{second_def_elementary}
		\bar{\Upsilon}_f[\Phi(\tau)] = \Upsilon_f[\tau].
	\end{equs}
	\end{itemize}
	For $u, v \in \CH$, smooth enough functions $\varphi, \psi$, one has
	\begin{equs} \label{key_identities_1}
		\bar{\Upsilon}_f[u \star v]\{ \phi \} 
		& = 
		\bar{\Upsilon}_f[u]  \circ \bar{\Upsilon}_f[v]\{ \phi \} ,
		\\	\label{key_identities_2}	\bar{\Upsilon}_f[ \Delta_\mu u] \{\phi \otimes \psi\}& = \bar{\Upsilon}_f[u] \{\phi \psi\}\,,
\end{equs}
where    $ \bar{\Upsilon}_f[u]\{\varphi\} $ is the extension of the elementary differentials to the ones composed with a smooth enough function $\varphi$. 
\end{assumption}
These two identities were given in a general context in \cite{KL23}. Let us mention that they are very close in spirit to the Newtonian maps introduced in \cite{Lejay}. These identities are checked in \cite{BBH26} for multi-indices. Other types of flow methods have been applied in the context of singular SPDEs. Let us summarise below the ideas behind the different flow approaches:
\begin{itemize}
	\item Flow along the time parameter: This is performed in \cite{B15} and \cite{BBH26}. It is not clear that such an approach will work in the context of singular SPDEs with Regularity Structures.
	\item Flow along a scale $ \lambda $ that appears in the kernel $K_{\lambda}$ in the mild formulation of the singular SPDEs. This is the flow approach of Pawel Duch \cite{Duc21} inspired by the Polchinski flow \cite{P84} used in QFT for renormalising Feynman diagrams. A solution theory with multi-indices has been derived for the generalised KPZ equation in \cite{CF24a}.
	\item Flow along a small parameter multiplying the non-linear interaction of the singular SPDEs. This is the strategy used in \cite{BOS25}.
\end{itemize}
Behind these various methods, one expects to find in the proofs, algebraic identities similar to Assumption~\ref{flow_condition} needed for the solution theory.

In \cite{GLMZ25}, the authors exhibit an algebraic assumption that is sufficient for performing a fixed point argument in the context of rough differential equations. It is mainly given by a cocycle condition on the Hopf algebra $\CH$. We recall it in the next assumption
\begin{assumption}[Algebraic fixed point]\label{assumption_fixed_point}
 The Hopf algebra $ \CH$  is equipped with  
\ $L_{\alpha}: \CH \rightarrow \CH$,
a family of linear maps indexed by $\alpha \in [d]$,  homogeneous of degree one, satisfying the cocycle condition
\begin{equs}\label{eq:one_cocycle}
	\Delta L_{\alpha} (u) = (\id \otimes L_{\alpha}) \Delta(u) + L_{\alpha} (u) \otimes \mathbf{1}.
\end{equs}
Here, $ \one $ is the unit of $H$. 
Moreover, one has a pairing $ \langle \cdot, \cdot  \rangle_{\CH} : \CH \otimes \CH \rightarrow $ such that
\begin{itemize}
	\item $ L_{\alpha}(\one) $ is a basis of $\CH_1$.
	\item For every $x, y \in \CH$, one has
	\begin{equation}
		\label{pairing_xy}
		\langle  L_{\alpha}(x) , L_{\beta}(y) \rangle_{\CH} = \delta_{\alpha}^{ \beta} \langle x,y \rangle_{\CH}.
		\end{equation}
\end{itemize} 
\end{assumption}

From \cite[Theorem 2]{CK1} on the universal property of $\mathcal{H}_\BCK$, the Butcher-Connes-Kreimer Hopf algebra, identity \eqref{eq:one_cocycle}, the existence of $1$-cocycle implies the existence of a unique  morphism $ \Lambda : \mathcal{H}_\BCK\to\CH$ sending $\mathcal{B}_+^\alpha$ to $L_{\alpha}$. We denote the dual map by $\Lambda^*: \CH^* \rightarrow \CH_{\BCK}^*$. Here, $ \CB^\alpha_+ $ takes a forest and connects all the roots to a new root labelled by $\alpha$. Now, we can state one of the main results of the present paper:
\begin{theorem} \label{main_thm_1} Assumption~\ref{assumption_fixed_point} implies Assumption~\ref{flow_condition}.
	\end{theorem}

	The fact that some combinatorial sets possess two solution theories, fixed point and flow, has been known since \cite{B15} (Shuffle Hopf Algebra) and \cite{B21} (Butcher-Connes-Kreimer Hopf algebra).

	Then, a natural question is to find a counter-example to the opposite direction proposed by the previous theorem. We are able to provide multi-indices as a counter-example via the second main result of the present paper. Indeed, one can have a morphism between $ \CH_{\BCK} ^*$ and $ \CH^* $ without \eqref{eq:one_cocycle}. This  fact has been observed for multi-indices and it is explained in the proof of Proposition
	 \ref{prop_flow_multi}.  

 \begin{theorem} \label{main_thm_2}
 	The multi-indices Hopf algebra does not satisfy Assumption~\ref{assumption_fixed_point}.
 \end{theorem} 
 This theorem is just a consequence of Theorem \ref{thm_cocycle_multi} which shows that the multi-indices do not satisfy \eqref{eq:one_cocycle} meaning that one cannot find a $1$-cocycle $L$ such that $L(\one)\neq0$.

 The two above-mentioned theorems can be summarised in the following way:  
 \textit{Decorated trees can be used for fixed point and flow approaches  whereas multi-indices works only for flow approaches.}
 The second part of this claim is not rigorous as the Assumption~\ref{assumption_fixed_point} is a sufficient but not a necessary assumption for getting the fixed point. But we still get a mathematical theorem on what has been observed in practice (Only flow approaches have been applied to multi-indices), that is showing that the algebraic assumption for the fixed point is not satisfied. We conclude this introduction with a couple of remarks before summarising the content of the paper.

 \begin{remark} One can get an analogue of Assumption~\ref{assumption_fixed_point} within the context of Regularity Structures. One expects to get a deformed cocycle property in the sense that there exists  
 	$L_{\alpha}: \CH \rightarrow \CH$,
 	a family of linear maps indexed by $\alpha \in \mathbb{N}^{d+1}$,  homogeneous of degree one, satisfying the deformed cocycle condition
 	\begin{equs}\label{eq:one_cocycle_reg}
 		\Delta L_{\alpha} (u) = (L_{\alpha} \otimes \id) \Delta(u) + \sum_{\ell \in \mathbb{N}^{d+1}} \frac{X^{\ell}}{\ell!} \otimes L_{\alpha + \ell} (u).
 	\end{equs}
 	where now monomials of the form $X^{\ell}$ are part of $\CH$.
 	A similar expression is given in \cite[Proposition 4.17]{BHZ} but used for the recentering and it appears for the first time in \cite[Section 8]{reg} expressed for a different decorated trees basis.
 	In the identity \eqref{eq:one_cocycle_reg}, one has an infinite sum but, in practice, it is finite as the  decorations $ \alpha + \ell $ cannot be arbitrarily large. One can also make sense of this infinite sum by using a bigrading introduced in \cite[Section 2.3]{BHZ}. 
Then, one has to use this condition to show the fixed point via Regularity Structures following \cite{BCCH} where a B-series formalism is used for the solution expansion.

The identity \eqref{key_identities_1} can be proved the same way as in \cite[Proposition 2.2]{BB26} which gives a morphism property for $ \star $ dual of $\Delta$. One has to use the explicit formula for the product $ \star $ given in \cite[Proposition 3.17]{BM22}. The other identity \eqref{key_identities_2} which corresponds to a Leibniz rule should be more straightforward. The fact that Regularity Structures multi-indices will not satisfy Assumption \ref{assumption_fixed_point} will follow the same argument as for rough paths multi-indices. Therefore, one gets a theoretical argument regarding the failure of finding a fixed point for the multi-indices for singular SPDEs.
 \end{remark}
 
 \begin{remark}
 	The paper \cite{GLMZ25} only covers the case where the equation \eqref{RDE}  takes values in $ \mathbb{R}^{m} $. In \cite{Manchon20}, the authors consider equations taking values in homogeneous space which requires non-commutative Hopf algebras. We expect similar assumptions to work in this case but one has to check analytically that a variant of Assumption \ref{assumption_fixed_point} provides a fixed point in this context.
 	The same is true for the Assumption \ref{flow_condition} as in \cite{KL23} the manifold case is treated by using charts. 
 	\end{remark}
 	
 	Let us outline the paper by summarising the content of its sections. In Section \ref{Sec::2}, we recall the derivation of the notion ``rough paths", and motivate the Hopf algebra and the one-cocycle condition \eqref{eq:one_cocycle}. Moreover, we briefly recall important concepts in Butcher-Connes-Kreimer Hopf algebra since the proof of our main theorem is based on the morphism sending $\CH^*$ to $\CH_{\BCK}^*$. 
 	
 	In Section \ref{Sec::3}, we use the universal result from Theorem \ref{thm:univBCK} to get the existence of a morphism $ \Lambda :  \CH_{\BCK} \rightarrow \CH_{} $ from the $1$-cocycle of $\CH$. Then, we are able in Proposition \ref{symmetry_H} to define a $\ST$-B-series from the B-series on the trees. We follow by proving a formula for the  composition law of $\ST$-B-series in Corollary \ref{composition_H} that proceeds from the one on the trees in Proposition \ref{prop:composition}. We finish the section by introducing controlled rough paths (see Definition \ref{control_rough_path}) and we show that the B-series based on the rough path $\mathbf{X}$ is a controlled rough path in Corollary \ref{corollary:controlled_RP}.
 	
 		In Section \ref{Sec::4}, we prove that the elementary differentials of $\ST$-B-series obtained in Section \ref{Sec::3} satisfy the two key identities in Assumption \ref{flow_condition}. This is a consequence of Proposition \ref{flow_morphism} that uses deeply the results on $ \CH_{\BCK} $ (see Propositions \ref{morphism_element} and \ref{prop:key_indentity_2}). We conclude this section by the proof of Theorem \ref{main_thm_1}.
 		
 			In Section \ref{Sec::5}, we discuss the solution theory of RDEs driven by multi-indices rough paths. We show in Theorem \ref{thm_cocycle_multi} that the multi-indices Hopf algebra does not satisfy the one-cocycle condition which is a natural sufficient condition in implementing the fixed point argument for RDEs solutions. This somehow answers the question in the literature that why it seems to be impossible to find the fixed point for multi-indices RDEs.  We finish the section by showing that multi-indices satisfy Assumption \ref{flow_condition} in Proposition \ref{prop_flow_multi}.

 \subsection*{Acknowledgements}
 
 {\small
 	Y. B., Y. H., P. L. gratefully acknowledge funding support from the European Research Council (ERC) through the ERC Starting Grant Low Regularity Dynamics via Decorated Trees (LoRDeT), grant agreement No.\ 101075208. Views and opinions expressed are however those of the author(s) only and do not necessarily reflect those of the European Union or the European Research Council. Neither the European Union nor the granting authority can be held responsible for them. 
 	This project was started the week of the 6th October 2025 when Zhicheng Zhu visited the Université de Lorraine in Nancy. This visit was funded by the ERC LoRDeT.

 } 
 
\section{Rough paths and Hopf algebras}
\label{Sec::2}
\subsection{Rough paths}
Let us briefly recall the derivation the notion rough paths and their associated Hopf-algebraic definition from the RDEs' point of view. Recall that the target RDE \eqref{RDE} is
\begin{equs}
	\rd Y_t =  \sum_{\alpha \in [d]} f_\alpha(Y_t) \rd X^\alpha_{t}
\end{equs}
where $[d]:=\{1, \ldots, d\}$ and $t \in [0,T]$.
Suppose the integration with respect to $X_t$ is well-defined. Then, the equation can be rewritten in the  integral form
\begin{equs}
	\rd Y_t =   Y_s + \sum_{\alpha \in [d]}  \int_{s}^t f_\alpha(Y_{r_1}) \rd X^\alpha_{r_1}
\end{equs}
for any $0 \le s \le t \le T$. Taylor expansion of $f_a(Y_{r_1})$ around $Y_s$ yields 
\begin{equs}
	\rd Y_t =   Y_s + \sum_{\alpha \in [d]}  
	\sum_{n \in \mathbb{N}} \sum_{b \in [m]^n } 
	\frac{1}{n!}
	\int_{s}^t 
	\prod_{i =1}^n\partial_{b_i} f_\alpha(Y_{s})(Y_{r_1}^{b_i}-Y_{s}^{b_i}) \rd X^\alpha_{r_1} 
\end{equs}
where $\partial_{b_i}$ is the partial derivative in the direction of the $b_i$-th variable. Notice that we can repeat the ``integration-Taylor procedure" for $Y_{r_1}$ and get
\begin{equs} \label{eq:Taylor}
	\rd Y_t &=   Y_s 
	+ 
	\sum_{\alpha \in [d]}   f_\alpha(Y_{s}) \int_{s}^t \rd X^a_{r_1} 
	\\&+
	\sum_{\alpha \in [d]}  
	\sum_{\beta \in [d]}
	\sum_{b \in [m] } 
	f_\beta^b\partial_{b} f_a(Y_{s})
	\int_{s}^t\int_{s}^{r_1} \rd X^\beta_{r_2}\rd X^\alpha_{r_1}  
	\\&+
	\sum_{\alpha \in [d]}  
	\sum_{\beta \in [d]^2}
	\sum_{b \in [m]^2 } 
	f_{\beta_1}^{b_1}f_{\beta_2}^{b_2}\partial_{b_1b_2} f_a(Y_{s})
	\int_{s}^t\int_{s}^{r_1} \rd X^{\beta_1}_{r_2}   \int_{s}^{r_1} \rd X^{\beta_2}_{r_3} \rd X^\alpha_{r_1}  
	+ R
\end{equs}
where $f_\beta^b$ is the $b$-th entry of $f_\beta$, $\partial_{b_1b_2}$ is a shorthand notation of $\partial_{b_1}\partial_{b_2}$, and the remainder $R$ can be made explicit by iterating the integration-Taylor procedure.
One observes that if we want to formally and algebraically encode such expansion of the solution via a Hopf algebra $\CH$, there are three main ingredients that need to be defined: 
\begin{itemize}
	\item the multiplication of iterated integrals (from monomials in the Taylor expansion),
	\item the composition of integrals, which satisfies the Chen's relation \cite{Chen},
	\item the iteration that preserves necessary properties of iterated integrals.
\end{itemize}
The first two notions lead to Definition \ref{def:rough path} of rough paths, and the third property is ensured by the one-cocycle condition \eqref{eq:one_cocycle} which links elements in $\CH$ to tree structures.
%
%Suppose that we have two Hopf algebras $\CH(\mu, \Delta)$ and $\CH^*(\star, \Delta_\mu)$ which are dual to each other. Let us encode the iterated integrals in the expansion solutions by $\langle \mathbf{X}, \Ch \rangle$ with
%\begin{equs}
%	\langle \mathbf{X}_{s,t}, \bullet_i \rangle = \int_s^t \rd X^i_r
%\end{equs}
%where $\mathbf{X} \in \CH^*$, $\Ch \in \CH$, and $\bullet_i$ are basis of $\CH_1$. Here, the integral is in the rough sense (It\^o, Stratonovich, etc.) since the classical integration is ill-defined for $X$. We further impose that $\star$ is adjoint to $\Delta$ and that  $\mu$ is adjoint to $\Delta_\mu$, i.e., for any $\Ch_1^*,\Ch_2^*,\Ch^* \in \CH^*$ and any $\Ch_1,\Ch_2,\Ch \in \CH$,
%\begin{equs}\label{eq:adjoint1}
%	\langle \Ch_1^* \star  \Ch_2^*, \Ch\rangle = \langle \Ch_1^* \otimes \Ch_2^*, \Delta \Ch\rangle
%\end{equs}
%and 
%\begin{equs}\label{eq:adjoint2}
%	\langle \Ch^*, \mu (\Ch_1, \Ch_2)\rangle = \langle \Delta_\mu \Ch^*, \Ch_1 \otimes \Ch_2\rangle.
%\end{equs}
%Then, we can use $\mu$ to represent the multiplication of iterated integrals, and use $\star$ to encode the composition of them, which leads to the following definition.
Since classical iterated integrals of $X$ is ill-defined, one needs to find some stochastic integration such as It\^o or Stratonovich integrals to make sense of the rough integrals. However, the previously mentioned properties of iterated integrals are needed to be preserved, which leads to the notion of rough paths.

Suppose $\CH(\mu, \Delta) = \bigoplus_{n\in \mathbb{N}}\mathcal{H}_n$ and $\CH^*(\star, \Delta_\mu)= \bigoplus_{n\in \mathbb{N}}\mathcal{H}^*_n$ are a graded Hopf algebra and its graded dual Hopf algebra. Further denote $|h|$ the grade of $h \in \CH$, i.e., $h \in \mathcal{H}_{|h|}$.
%
%Suppose that we have two Hopf algebras $(\CH,\mu, \Delta)$ and $(\CH^*,\star, \Delta_\mu)$ which are dual to each other meaning that we are given a pairing $ \langle \cdot, \cdot \rangle $ such that
% $\star$ is adjoint to $\Delta$ and that  $\mu$ is adjoint to $\Delta_\mu$. One has for any $h_1^*,h_2^*,h^* \in \CH^*$ and any $h_1,h_2,h \in \CH$,
%\begin{equs}\label{eq:adjoint1}
%	\langle h_1^* \star  h_2^*, h\rangle = \langle h_1^* \otimes h_2^*, \Delta h\rangle
%\end{equs}
%and 
%\begin{equs}\label{eq:adjoint2}
%	\langle \Ch^*, \mu (\Ch_1, \Ch_2)\rangle = \langle \Delta_\mu \Ch^*, \Ch_1 \otimes \Ch_2\rangle.
%\end{equs}
%
% Let us encode the iterated integrals in the expansion solutions by $\langle \mathbf{X}, \Ch \rangle$ with
%\begin{equs}
%	\langle \mathbf{X}_{s,t}, \bullet_i \rangle = \int_s^t \rd X^i_r
%\end{equs}
%where $\mathbf{X} \in \CH^*$, $\Ch \in \CH$, and $\bullet_i$ are basis of $\CH_1$. Here, the integral is in the rough sense (It\^o, Stratonovich, etc.) since the classical integration is ill-defined for $X$. 
%Then, we can use $\mu$ to represent the multiplication of iterated integrals, and use $\star$ to encode the composition of them, which leads to the following definition.
\begin{definition}[Rough paths] \label{def:rough path}
	A $\CH$-rough path of regularity $\gamma$ is a map $\mathbf{X} : [0, T] \times [0, T] \rightarrow \CH^*$ satisfying the following three
	conditions
	\begin{enumerate}
		\item For every $h_1, h_2 \in \CH$
		\begin{equs}\label{eq:multiplication_integral}
			\langle \mathbf{X}_{s,t}, \mu(h_1,  h_2) \rangle := \langle \mathbf{X}_{s,t}, h_1 \rangle \cdot \langle \mathbf{X}_{s,t}, h_2 \rangle.
		\end{equs}
		\item (Chen's relation) For any $0\le s \le r \le t \le T$,
		\begin{equs}\label{eq:Chen}
			\mathbf{X}_{s,t} = \mathbf{X}_{s,r} \star \mathbf{X}_{r,t}.
		\end{equs}
		%		$\star$ is determined by $\mu$
		%		as the coproduct $\Delta_\mu$ in $\CH^*(\star,\Delta_\mu)$ is the adjoint of $\mu$.
		%		$\star$ and $\Delta_\mu$ satisfies the property of a bialgebra.
		\item For any $h \in \mathcal{H}_{|h|}$
		\begin{equs}
			\sup_{s \ne t}  \frac{|\langle\mathbf{X}_{s,t}, h \rangle|}{|t-s|^{\gamma|h|}} < \infty.
		\end{equs}
	\end{enumerate}
\end{definition}
\begin{remark}
 In practice, $\mu$ is determined by the property of the chosen rough integration. For example, if the rough integral satisfies integration by parts, the shuffle product is chosen. 
 
 The Chen's relation can be expressed equivalently in the dual way
		\begin{equs}
			\langle	\mathbf{X}_{s,t}, h \rangle = \langle \mathbf{X}_{s,r} \otimes \mathbf{X}_{r,t}, \Delta h \rangle = \langle \mathbf{X}_{s,r}, h^{(1) }\rangle\langle \mathbf{X}_{s,r}, h^{(2) }\rangle
		\end{equs}
		where we use the Sweedler notation for $\Delta$ given by
		\begin{equation*}
			\Delta h  = \sum_{(h)}  h^{(1)} \otimes h^{(2)}. 
		\end{equation*}
\end{remark}

\subsection{Butcher-Connes-Kreimer Hopf algebra and branched rough paths}

The Butcher-Connes-Kreimer Hopf algebra $\CH_{\BCK}(\mathrm{Span}_{\mathbb{R}}(\CF), \odot, \Delta_{\BCK})$, is defined on the linear span of forests of non-planar rooted trees. Here, $\odot$ is the commutative forest product which represents the juxtaposition of rooted trees. We denote the set of forests by $\CF$ and the set of rooted trees by $\CT$. $\CH_{\BCK}$ is equipped with $1$-cocycle $(\CB^{\alpha}_+)_{\alpha \in [d]}$ which is a family linear maps. Consider a forest $\prod_{i}^{\odot n} \tau_i$, $\CB_{+}^{\alpha}\left(\prod_{i}^{\odot n} \tau_i\right)$ amounts to grafting $\tau_i$ to a common root meaning that all the roots of the $\tau_i$ are connected via new edges to a new root decorated by $\alpha$.
\begin{example} We provide below an example of computation with the map $ \mathcal{B}_+^\alpha $
\begin{equs}
	\CB_{+}^{\alpha}\left(
	\begin{tikzpicture}[scale=0.2,baseline=0.1cm]
		\node at (0,1)  [dot,label= {[label distance=-0.2em]above: \scriptsize  $ \beta $} ] (root) {};
	\end{tikzpicture}
	\odot
\begin{tikzpicture}[scale=0.2,baseline=0.1cm]
	\node at (0,1)  [dot,label= {[label distance=-0.2em]below: \scriptsize  $ \gamma $} ] (root) {};
	\node at (0,3)  [dot,label={[label distance=-0.2em]right: \scriptsize  $ \delta $}] (center) {};
	\draw[kernel1] (center) to
	node [sloped,below] {\small }     (root);
\end{tikzpicture}
	\right)
	= \begin{tikzpicture}[scale=0.2,baseline=0.1cm]
	\node at (0,1)  [dot,label= {[label distance=-0.2em]below: \scriptsize  $ \alpha $} ] (root) {};
	\node at (1,3)  [dot,label={[label distance=-0.2em]right: \scriptsize  $ \gamma $}] (right) {};
	\node at (1,5)  [dot,label={[label distance=-0.2em]right: \scriptsize  $ \delta $}] (rightc) {};
	\node at (-1,3)  [dot,label={[label distance=-0.2em]above: \scriptsize  $ \beta $} ] (left) {};
	\draw[kernel1] (right) to
	node [sloped,below] {\small }     (root);
	\draw[kernel1] (right) to
	node [sloped,below] {\small }     (rightc); \draw[kernel1] (left) to
	node [sloped,below] {\small }     (root);
	\end{tikzpicture}.
\end{equs}
\end{example}
We identify $\CH_{\BCK}$ with its dual $\CH^*_{\BCK}$ via the basis of forest from $  \CH_{\BCK}$.
\begin{definition}[$\CH_\BCK$ perfect pairing]\label{dfn:perfBCK}  
	The perfect pairing of $\CH_\BCK$ is
	\begin{equs}\label{eq:pairing}
		&\langle \cdot, \cdot \rangle: \CH^*_{\BCK} \times \CH_{\BCK} \rightarrow \mathbb{R}
		\\&
		\langle \Cf_1^*, \Cf_2 \rangle := \delta^{\Cf_1}_{\Cf_2} S(\Cf_2), \quad\quad\text{ for any $\Cf_1^* \in \CF^*$ and any $\Cf_2 \in \CF$}
	\end{equs}
	where $\delta$ is the Kronecker delta and $S$ is the symmetry factor in Definition \ref{def:symmetry_factor}. 
\end{definition}
Due to this identification, we will make the following abuse of notations in the sequel  $\tau, \Cf$ for $\tau^*, \Cf^*$. Branched rough paths introduced in \cite{Gub06} live in  $\CH_{\BCK}^*(\mathrm{Span}_{\mathbb{R}}(\CF), \star_{\BCK}, \Delta_{\shuffle})$, the dual of $\CH_{\BCK}$. The unshuffle coproduct runs over the partition of forest into two, and trees are primitive elements of $\Delta_{\shuffle}$, i.e., for any $\tau, \sigma \in \CT$
\begin{equs}
	&\Delta_{\shuffle} \tau = \mathbf{1} \otimes \tau + \tau \otimes \mathbf{1}, \quad 
	\Delta_{\shuffle} (\tau \odot \sigma) = \Delta_{\shuffle} \tau \odot \Delta_{\shuffle} \sigma,
\end{equs}
where $\mathbf{1}$ is the trivial (empty) forest, and $$(h_1 \otimes h_2) \odot (h_3 \otimes h_4) :=  (h_1 \odot h_3) \otimes( h_2 \odot h_4).$$
The $\star_{\BCK}$ product is the Grossman-Larson product \cite{GL89} which can be also obtained through the Guin-Oudom construction \cite{Guin1,Guin2} on the pre-Lie grafting product $\triangleright_{\BCK}$ of rooted trees. We recall the definitions of $\triangleright_{\BCK}$ and $\star_{\BCK}$. 
\begin{definition}[Grafting product $\triangleright_{\BCK}$]
	For any $\sigma, \tau \in \CT$, the grafting product $\sigma \triangleright_{\BCK} \tau $ amounts to adding an edge connecting the root of $\sigma$ to one vertex of $\tau$ and summing over all vertices of $\tau$.
	\begin{equation} \label{formula_grafting}
		\sigma \triangleright _{\BCK}\tau = \sum_{v \in \CV_\tau} \sigma \triangleright_v \tau ,
	\end{equation}
	where $\CV_\tau$ is the vertex set of $\tau$ and $\sigma \triangleright_v \tau$ adds an edge between the root of $\sigma$ and the vertex $v$.
\end{definition}
\begin{example}[Grafting product $\triangleright_{\BCK}$]
	Here, we give an example for the grafting product. We omit the decorations on the nodes.
	\begin{equs}
		{\color{blue}\bullet}
		\triangleright_{\BCK}
		\begin{forest}
			[[][]]
		\end{forest}
		= 
		\begin{forest}
			[[,fill=blue, edge={blue, thick}][][]]
		\end{forest}
		+2
		\begin{forest}
			[[[,fill=blue, edge={blue, thick}]][]]
		\end{forest}
		,\quad
		{\color{blue}
			\begin{forest}
				[[]]
			\end{forest}
		}
		\triangleright_{\BCK}
		\begin{forest}
			[[][]]
		\end{forest}
		= 
		\begin{forest}
			[
			[ , fill=blue, edge={blue, thick}, for tree={fill=blue, edge={blue, thick}} 
			[ , ] 
			]
			[ , ]
			[ , ]
			]
		\end{forest}
		+2
		\begin{forest}
			[
			[ 	[ , fill=blue, edge={blue, thick}, for tree={fill=blue, edge={blue, thick}} 
			[  ] 
			] ]
			[  ]
			]
		\end{forest}
	\end{equs}
	where the coefficient $2$ appears since the trees are non-planar.
\end{example}

With the product $\star_{\BCK}$, $\CH^*_{\BCK}$ forms the universal envelop of the Lie algebra obtained from the antisymmetrisation of the pre-Lie product $\triangleright_{\BCK}$. Now, we introduce $\star_{\BCK}$ through the Guin-Oudom procedure, which constructs an algebra isomorphic to the Grossman-Larson algebra. Let us first define the simultaneous grafting product $\bstar_{\BCK}$, the reduced version of $\star_{\BCK}$.
\begin{definition}[Simultaneous grafting $\bstar_{\BCK}$]
	For any $\prod_{i=1}^{\odot n} \sigma_i \in \CF$ and $\tau \in \CT$, the simultaneous grafting reads
	\begin{equs}
		\prod_{i=1}^{\odot n } \sigma_i \bstar_\BCK \tau = \sum_{v_1,\ldots, v_{n} \in \CV_\tau} \sigma_1 \triangleright_{v_1} \ldots \sigma_n \triangleright_{v_{n}} \tau
	\end{equs}
	which means choose $n$ vertices of $\tau$ and add an edge connecting the root of $\sigma_i$ to $v_i$. The nodes $v_i$ are not necessarily distinct. Moreover, for $\tau_i \in \CT$
	\begin{equs}
		\sigma \bstar_{\BCK} \prod_{j=1}^{\odot m}  \tau_i = \sum_{i=1}^{n}\sigma \bstar_{\BCK} \tau_i \odot \prod_{j \neq i}^{\odot}  \tau_j
	\end{equs}
	which is analogue to the Leibniz rule.
\end{definition}
\begin{example}[Simultaneous grafting $\bstar_{\BCK}$] \label{eg:simultaneous_grafting} One has
	\begin{align*}
		&
		\left(
		{\color{red} \bullet} 
		\odot 	
		\begin{forest}
			[,fill=blue, edge={blue, thick}, for tree={fill=blue, edge={blue, thick}} []] 
		\end{forest}
		\right)
		\bstar_{\BCK}
		\left(
		\bullet
		\odot
		\begin{forest}
			[[][]]
		\end{forest}
		\right)
		\\=&
		\begin{forest}
			[	[,fill=red, edge={red, thick}]	[,fill=blue, edge={blue, thick}, for tree={fill=blue, edge={blue, thick}} []] ]
		\end{forest}
		\odot
		\begin{forest}
			[[][]]
		\end{forest}
		+
		\begin{forest}
			[	[,fill=red, edge={red, thick}]]
		\end{forest}
		\odot
		\begin{forest}
			[	[,fill=blue, edge={blue, thick}, for tree={fill=blue, edge={blue, thick}} []]  [][]]
		\end{forest}
		+2 \,
		\begin{forest}
			[	[,fill=red, edge={red, thick}]]
		\end{forest}
		\odot
		\begin{forest}
			[ [	[,fill=blue, edge={blue, thick}, for tree={fill=blue, edge={blue, thick}} []]][]]
		\end{forest}
		+
		\begin{forest}
			[[,fill=blue, edge={blue, thick}, for tree={fill=blue, edge={blue, thick}} []] ]
		\end{forest}
		\odot
		\begin{forest}
			[	[,fill=red, edge={red, thick}][][]]
		\end{forest}
		+
		2 \,
		\begin{forest}
			[[,fill=blue, edge={blue, thick}, for tree={fill=blue, edge={blue, thick}} []] ]
		\end{forest}
		\odot
		\begin{forest}
			[	[[,fill=red, edge={red, thick}]][]]
		\end{forest}
		\\+&
		\bullet \odot
		\begin{forest}
			[[,fill=red, edge={red, thick}]	[,fill=blue, edge={blue, thick}, for tree={fill=blue, edge={blue, thick}} []] [][]]
		\end{forest}
		+
		2
		\bullet \odot
		\begin{forest}
			[[,fill=red, edge={red, thick}]	[[,fill=blue, edge={blue, thick}, for tree={fill=blue, edge={blue, thick}} []] ][]]
		\end{forest}
		+
		2
		\bullet \odot
		\begin{forest}
			[[,fill=blue, edge={blue, thick}, for tree={fill=blue, edge={blue, thick}} []]	[ [,fill=red, edge={red, thick}]][]]
		\end{forest}
		+
		2
		\bullet \odot
		\begin{forest}
			[	[ [,fill=red, edge={red, thick}]][[,fill=blue, edge={blue, thick}, for tree={fill=blue, edge={blue, thick}} []]]]
		\end{forest}.
	\end{align*}
\end{example}
Through the Guin-Oudom construction, the $\star_{\BCK}$ product is obtained by partitioning the forest $\prod_{i=1}^{\odot n} \sigma_i$ into two parts and simultaneously grating one part to $\prod_{i=1}^{\odot m} \tau_i$ while putting the other part aside in the forest.
\begin{definition}[$\star_{\BCK}$]\label{def:star_BCK}
	For any $\Cf_1 , \Cf_2 \in \CF$. The product $\star_{\BCK}$ is defined as the following.
	\begin{equs}
	&	\Cf_1 \star_{\BCK}  \mathbf{1} = \mathbf{1} \star_{\BCK}	\Cf_1  =   \Cf_1, \quad \Cf_1  \star_{\BCK} 	\Cf_2 = \mathcal{M}_{\odot} \circ (\id \otimes \cdot \bstar_{\BCK} \Cf_2) \circ \Delta_{\shuffle} \Cf_1,
	\end{equs}
	where $ \mathcal{M}_{\odot} ( \Cf_1 \otimes \Cf_2 ) = \Cf_1 \odot \Cf_2$.
\end{definition}
\begin{example}[ $\star_{\BCK}$]
	\begin{align*}
		&
		\left(
		{\color{red} \bullet} 
		\odot 	
		\begin{forest}
			[,fill=blue, edge={blue, thick}, for tree={fill=blue, edge={blue, thick}} []] 
		\end{forest}
		\right)
		\star_{\BCK}
		\left(
		\bullet
		\odot
		\begin{forest}
			[[][]]
		\end{forest}
		\right)
		\\=&
		{\color{red} \bullet} 
		\odot 
		\left(\begin{forest}
			[,fill=blue, edge={blue, thick}, for tree={fill=blue, edge={blue, thick}} []] 
		\end{forest}
		\bstar_{\BCK} 
		\left(
		\bullet
		\odot
		\begin{forest}
			[[][]]
		\end{forest}
		\right)
		\right)
		+
		\begin{forest}
			[,fill=blue, edge={blue, thick}, for tree={fill=blue, edge={blue, thick}} []] 
		\end{forest}
		\odot
		\left({\color{red} \bullet}
		\bstar_{\BCK} 
		\left(
		\bullet
		\odot
		\begin{forest}
			[[][]]
		\end{forest}
		\right)
		\right)
		+
		\left(
		{\color{red} \bullet} 
		\odot 	
		\begin{forest}
			[,fill=blue, edge={blue, thick}, for tree={fill=blue, edge={blue, thick}} []] 
		\end{forest}
		\right)
		\bstar_{\BCK}
		\left(
		\bullet
		\odot
		\begin{forest}
			[[][]]
		\end{forest}
		\right)
	\end{align*}
	where $	\left(
	{\color{red} \bullet} 
	\odot 	
	\begin{forest}
		[,fill=blue, edge={blue, thick}, for tree={fill=blue, edge={blue, thick}} []] 
	\end{forest}
	\right)
	\bstar_{\BCK}
	\left(
	\bullet
	\odot
	\begin{forest}
		[[][]]
	\end{forest}
	\right)$
	is shown in Example \ref{eg:simultaneous_grafting}.
\end{example}

Since $\CH_{\BCK}$ is based on non-planar trees, we have to introduce the symmetry factors of trees and forests, which play important roles in both the pairing (inner product) and B-series for solutions.
\begin{definition}[Symmetry factor] \label{def:symmetry_factor}
%	The symmetry factor of $\Cf \in \CF$ is defined as the cardinal of its automorphism group
%	\begin{equs}
%		S(\Cf) = |\mathrm{Aut}(\Cf)|.
%	\end{equs}
%	Explicitly, for 
The symmetry factor of a forest $\Cf$ is the number of permutations of vertices of $\Cf$ letting $\Cf$ globally invariant.
 For any $\tau = \CB_+^{\alpha} \left(\prod_{i}^{\odot n_\tau} \tau_i \right)\in \CT$
	\begin{equs}
		S(\tau) = \prod_{j=1}^c  r_j! S(\tau_j)^{r_j}
	\end{equs}
	where $c$ is the number of isomorphic classes among $\tau_i \in \CT$, and $r_j$ is the cardinal of the $j$-th class  in which elements are isomorphic to $\tau_j$. For any $ \Cf = \prod_{i=1}^{\odot \, \mathrm{Card}(\Cf)} \sigma_i\in \mathcal{F}$ with $  \mathrm{Card}(\Cf) $ the number of trees in $\Cf$, one has
	\begin{equs}
		S(\Cf) = \prod_{j=1}^{c_{\Cf}}  r^{(\Cf)}_j! S(\sigma_j)^{r^{(\Cf)}_j},
	\end{equs}
	where  $c_{\Cf}$ is the number of isomorphic classes among $\sigma_i \in \CT$, and $r_j^{(\Cf)}$ is the cardinal of the $j$-th class  in which elements are isomorphic to $\sigma_j$. 
\end{definition}

\begin{example} Below, we compute the symmetry factors for trees where we have assumed that all the nodes decorations are the same.
	\begin{align*}
%		&S\left(
%		\begin{forest}
%			[[][]] 
%		\end{forest}
%		\right) = 2,
%		\quad
		S\left(
		\begin{forest}
			[[[]][]] 
		\end{forest}
		\right) = 1,
		\quad
		S\left(
		\begin{forest}
			s sep=1pt,
			l sep=1pt
			[
				[,s sep=1pt,
				l sep=1pt
					[][][]
				]
				[,s sep=1pt,
				l sep=1pt
					[][][]
				]
			] 
		\end{forest}
		\right) = 2!(3!)^2=72, \quad
%		\\&
		S\left(
		\begin{forest}
			[[][]]
		\end{forest}
		\odot
		\begin{forest}
			[[][]]
		\end{forest}
		\right)
		= 2!\cdot2^2 =8.
%		\\&
%		S\left(
%		\begin{forest}
%			[[][]]
%		\end{forest}
%		\odot
%		\begin{forest}
%			[[[][][]][[][][]]] 
%		\end{forest}
%		\odot
%		\begin{forest}
%			[[[][][]][[][][]]] 
%		\end{forest}
%		\right) = 2\cdot2!(2!3!3!)^2 = 20736
	\end{align*}
\end{example}
The coproduct
	$\Delta_{\BCK}$ and the product $\star_{\BCK}$ are adjoint under the pairing \eqref{eq:pairing}, i.e, for any $h_1,h_2 \in \CH^*_{\BCK}$ and $h \in \CH_{\BCK}$.
	\begin{equation*}
		\langle h_1 \star_{\BCK} h_2, h \rangle = \langle h_1 \otimes h_2, \Delta_{\BCK} h \rangle. 
	\end{equation*}
\begin{remark}\label{remark:pairing}
	There is another pairing in the literature (see \cite{HK15}), which is
	$$(h_1^*, h_2) = \delta_{ h_2}^{h_1}.$$
 One can notice that the duality
	\begin{equs}
		(h_1 \star_{\BCK}  h_2, h) = (h_1 \otimes h_2, \Delta_{\BCK} h)
	\end{equs}
	forces the coefficients resulting from the product $\star_{\BCK}$ to be equal to those from $\Delta_{\BCK}$. 
	Consequently, we lose the relatively natural formula of $\star_{\BCK}$ where the identity defining the grafting \eqref{formula_grafting} has to be changed.
\end{remark}
Due to the definition of the inner product \eqref{eq:pairing}, an element in $\mathbf{X} \in \CH_{\BCK}^*$ has the general form
\begin{equs}
	\mathbf{X} = \sum_{\Cf \in \CF} \frac{\langle \mathbf{X}, \Cf \rangle }{S(\Cf)} \Cf^*.
\end{equs}
Now, to construct Branched rough paths that solve RDEs \eqref{RDE}, we define recursively $\mathbf{X}$ to satisfy simultaneously, for $h, h_1, h_2 \in \CH_{\BCK}$
\begin{equs}\label{eq:H_rough_path}
	\begin{aligned}
	&
	\langle \mathbf{X}_{s,t}, \bullet_\alpha \rangle := \int_{s}^{t} \rd X_r^\alpha =  X_t^\alpha -  X_s^\alpha,
	\\&
		\langle \mathbf{X}_{s,r}, h_1 \odot h_2 \rangle := \langle \mathbf{X}_{s,r}, h_1 \rangle \cdot \langle \mathbf{X}_{s,r}, h_1 \rangle,
	\\&
	\langle \mathbf{X}_{s,t},\CB_+^{\alpha}(h) \rangle := \int_{s}^{t} \langle \mathbf{X}_{s,r}, h \rangle \rd X_r^\alpha.
	\end{aligned}
\end{equs}
Here, the integral is in the rough sense (It\^o, Stratonovich, etc.). One can check that this definition together with $\star_{\BCK}$ satisfy Definition \ref{def:rough path} and form a rough path.

From the expansion \eqref{eq:Taylor} of solutions of RDEs, one notices that there are two main ingredients: Iterated integrals of the signals $X$, and the monomials of derivatives of the vector field $f$. The iterated integrals are represented by branched rough paths through the map $\mathbf{X}$ \eqref{eq:H_rough_path}. Therefore, it is left to find the algebraic expression of the monomials, which are called elementary differentials and are defined as following.
\begin{definition}[Elementary differentials] \label{def:elementary_differential}
	Elementary differentials are maps $\Upsilon_f: \CH_{\BCK}^* \times \mathbb{R}^m \rightarrow \mathbb{R}^m$, which are linear in $\CT^*$,  and have the following expression.
	\begin{equs}
		&\Upsilon_f\left[ \bullet_\alpha \right] (y)= f_\alpha (y),
		\\&
		\Upsilon_f\left[ \CB_{+}^{\alpha}	\left(\prod_{i=1}^{\odot n}  \tau_i\right) \right] (y):= \sum_{b \in [m]^n }  \left(\prod_{i=1}^n \partial_{b_i}f_\alpha (y)\right)  \prod_{i=1}^n\Upsilon_f^{b_i}[ \tau_i](y) ,
	\end{equs}
	where $\Upsilon_f^{b}$ is the $b$-th coordinate of $\Upsilon_f$. Outside the previous cases, the elementary differential is zero.
\end{definition}
%Then, putting the Taylor coefficients together with the $\SH$-rough paths, we have the following-defined $\ST$-Butcher series, which will be proven to be the expansion of solutions of RDEs later in Theorem \ref{thm:B-series_solution}.
%\begin{definition}[$\ST$-Butcher series]
%	$\ST$-Butcher series are maps: $\mathbb{R}^m \times \SH^* \rightarrow \mathbb{R}^m$ defined as 
%	\begin{equs}
%		B_{\!\ST}(y,\mathbf{Z}) = \langle \mathbf{Z}, \mathbf{1}  \rangle y + \sum_{\tau \in \ST} \frac{\Upsilon_f[\tau](y)}{S(\tau)} \langle \mathbf{Z}, \tau \rangle 
%	\end{equs}
%	where  $\mathbf{Z} \in \CH^*(\mu, \Delta)$ is a character which means
%	$\langle \mathbf{Z} , \prod^{(\mu) n}_{i=1} \tau_i\rangle = \prod^{n}_{i=1}\langle \mathbf{Z} , \tau_i \rangle$ 
%	and it satisfies $\langle \mathbf{Z} , \mathbf{1}  \rangle =1$.
%\end{definition}

 According to \cite[Theorem 5.1]{Gub06}, 
	given the signal $X$ and its associated branched rough path $\mathbf{X} \in \CH^*_{\BCK}$,
	the solution of equation \eqref{RDE}
	admits the B-series expression
	\begin{equs}\label{eq:Bseries_commutative}
		Y_t = B_{\CT}(Y_s,\mathbf{X}_{s,t} ) \end{equs}
		where one has for $\mathbf{Z} \in \CH_{\BCK}^*$
	\begin{equs} \label{B_series_BCK}
		 B_{\CT}(y,\mathbf{Z} ) = y +	\sum_{\tau \in \CT} \frac{\Upsilon_f[\tau](y)}{S(\tau)}\langle \mathbf{Z}, \tau \rangle.
	\end{equs}

\section{B-series representation of cocycle-RDE solutions}
\label{Sec::3}
It is shown in \cite{GLMZ25} that Assumption \ref{assumption_fixed_point} together with the smoothness assumption on the vector field $f$ is sufficient to perform a fixed point argument for proving the existence and uniqueness of solutions of RDEs \eqref{RDE}. In the sequel, we call the RDEs driven by rough paths associated to a one-cocycle-Hopf algebra cocycle-RDEs.

In this section, we will show that the solution of cocycle-RDEs admits a B-series representation and the composition law of B-series ensures that their coefficients form controlled rough paths. Notably, according to the universal property of $\CH_{\BCK}$ (Theorem \ref{thm:univBCK}), the one-cocycle condition leads to a morphism $\Lambda: \CH^* \rightarrow \CH_{\BCK}^*$. Therefore, properties of $\CH_{\BCK}$ elementary differentials are passed to those of $\CH$. The algebraic properties recalled in this section correspond to Assumption~\ref{assumption_fixed_point}.

\subsection{1-cocycle Hopf algebra and $\CH_{\BCK}$}
We have seen the need of iterations preserving properties of iterated integrals for solving the RDE \eqref{RDE}. This motivates the properties of the Hopf algebra $\CH$ which provides a sufficient condition in proving the fixed point argument.

Let the graded reduced Hopf algebra $\CH (\mu, \Delta)$  equipped with  
$L_{\alpha}: \CH \rightarrow \CH$,
a family of linear maps indexed by $\alpha \in [d]$,  homogeneous of degree one, satisfying the cocycle condition for $ h \in \CH $
\begin{equs}\label{eq:one_cocycle_bis}
	\Delta L_{\alpha} (h) = (\id \otimes L_{\alpha}) \Delta(h) + L_{\alpha} (h) \otimes \mathbf{1}.
\end{equs} 
%\begin{assumption}\label{assumption}
%	We also assume that $\SH$ has the following two properties.
%	\begin{itemize}
	%		\item  The collection $\{L_\alpha(\mathbf 1),\, \alpha=1,\ldots,d\}$ is a basis of $\SH_1$.
	%		\item  There exist a surjective morphism from $\CH_{\BCK}$ to $\SH$, which sends $\CB^\alpha_+$ to $L_\alpha$ for all $\alpha$.
	%		{\color{red} define $\CB_+$ at the beginning}
	%	\end{itemize} 
%\end{assumption}
By reduced, we mean that $ \CH_0 = \mathbb{R} \one $ where $\one $ is the unit for $\CH$.
Let us recall \cite[Theorem 2]{CK1} on the universal property of $\mathcal{H}_\BCK$.

\begin{theorem}\label{thm:univBCK}
	The Hopf algebra $\mathcal{H}_\BCK$ endowed with the cocycle $\mathcal{B}_+$ is initial in the category of Hopf algebras endowed with a $1$-cocycle.
	In plain words, any Hopf algebra $\CH$ endowed with a $1$-cocycle $L$ admits a unique morphism $\Lambda:\mathcal{H}_\BCK\to\CH$ sending $\mathcal{B}_+$ to $L$. It is naturally generalised to a collection of cocycles $L_{\alpha}$ by decorating the vertices.
\end{theorem}

By Theorem~\ref{thm:univBCK}, we have a unique morphism of Hopf algebra:
$$\Lambda:(\CH_\BCK,\odot,\Delta_\BCK)\to(\CH,\mu,\Delta)$$
such that $\Lambda \circ\CB_+^\alpha=L_\alpha\circ\Lambda$.
This morphism induces a dual morphism:
$$\Lambda^*:(\CH^*,\star,\Delta_\cdot)\to(\CH^*_\BCK,\star_{\BCK},\Delta_\shuffle).$$
The space $\CH$ split as $\CH=\CI\oplus\CK$, with $\CI=\Im(\Lambda)$, and $\CK=\Im(\Lambda)^\perp$ the orthogonal of $\Im(\Lambda)$.
Moreover, $\CI$ is closed by $\mu$, and $(L_\alpha)_{\alpha\in [d]}$, to be more precise, this is actually the closure of $\CH_0$ by $\mu$, and $(L_\alpha)_{\alpha\in [d]}$.
We denote by $\ST$ an orthogonal basis of $\CI$, we naturally identify $\CI$ and $\CI^*$ via this basis.

\begin{proposition} \label{symmetry_H}
	For every $ h \in \ST $, we define the norm $\Vert \cdot \Vert_{\CH}$ by
\begin{equation} \label{symmetry_factor_H}
	\Vert h \Vert_{\CH}^2 := \langle h, h \rangle_{\mathcal{H}}.
\end{equation}
For every $ \mathbf{X} \in \CI^* $, one has
\begin{equation} \label{basis_dual_H}
	\mathbf{X} = \sum_{h \in \ST} \frac{\langle \mathbf{X}, h \rangle_\CH}{\Vert h \Vert_{\CH}^2} h.
\end{equation}
%{\color{red} Should I change $S(h)$ to $\Vert h \Vert_{\CH}$ since the symmetry factor is different from the BCK one?}
%	{\color{magenta} PL: I prefer the notation $\Vert h \Vert_{\CH}^2$ (it's the square of the norm).}
Moreover, one can define the elementary differential of $h \in \ST$ by
\begin{equation} \label{elementary_diff_H}
	\bar{\Upsilon}_f(h) := \Upsilon_f(\Lambda^*(h))
\end{equation}
and the following $\ST$-Butcher series by
\begin{equation} \label{B_series_H}
	B_{\ST}(y,\mathbf{Z} ) := 	B_{\CT}(y,\Lambda^*(\mathbf{Z}) ).
\end{equation}
Then one has
\begin{equation} \label{formula_B_series_H}
	B_{\ST}(y,\mathbf{Z} )  = y +	\sum_{h \in \ST} \frac{\bar{\Upsilon}_f[h](y)}{\Vert h \Vert_{\CH}^2} \langle \mathbf{Z}, h \rangle_\CH.
\end{equation}
%\begin{equs}
%	\label{B_series_H}
%	B_{\SF}(y,\mathbf{Z} ) & = y +	\sum_{h \in \SF} \frac{\hat{\Upsilon}_f[h](y)}{S(h)} \langle \mathbf{Z}, h \rangle_\CH
%	\\ &=  y +	\sum_{\tau \in \CF} \frac{\Upsilon_f[\tau](y)}{S(\tau)}\langle \mathbf{Z}, \Lambda(\tau) \rangle_\CH
%	\\ &=  y +	\sum_{\tau \in \CT} \frac{\Upsilon_f[\tau](y)}{S(\tau)}\langle \mathbf{Z}, \Lambda(\tau) \rangle_\CH
%	\\ & = 	B_{\CT}(y,\Lambda^*(\mathbf{Z}) )
%\end{equs}
\end{proposition}
\begin{proof}
	The map $(B_\ST(y,\cdot)-y)$, and $(B_\CT(y,\cdot)-y)$ are linear, moreover, the former is exactly the elementary differential $\Upsilon_f$ as shown by the formula:
	\begin{equs}
		y+\Upsilon_f(\mathbf{X})  &= y +\sum_{\tau \in \CT} \frac{\Upsilon_f[\tau]}{S(\tau)} \langle \mathbf{X}, \tau \rangle.
		\\ &=B_{\CT}(y,\mathbf{X} ) . 
	\end{equs}
	Let us write $B_\ST(y,\mathbf{Z})-y$ as the composition $(B_\CT(y,\cdot)-y)\circ\Lambda^*$ evaluated at $\mathbf{Z}$.
	In the basis $\ST$, one has
	\begin{equs}
		B_{\ST}(y,\mathbf{Z} )  &= 
%		y+\Upsilon_f(\mathbf{Z}) = y+ \Upsilon_f\left(\sum_{h \in \ST} \frac{\langle \mathbf{Z}, h\rangle_{\CH}}{\Vert h \Vert_{\CH}^2} h\right)
		y+B_{\ST}\left(y,\sum_{h \in \ST} \frac{\langle \mathbf{Z}, h\rangle_{\CH}}{\Vert h \Vert_{\CH}^2} h\right)
		\\  &= 
		y +	\sum_{h \in \ST} \frac{(B_\ST(y,\cdot)-y)(h)}{\Vert h \Vert_{\CH}^2} \langle \mathbf{Z}, h \rangle_\CH.
		\\  &= y +	\sum_{h \in \ST} \frac{(B_\CT(y,\cdot)-y)\circ\Lambda^*(h)}{\Vert h \Vert_{\CH}^2} \langle \mathbf{Z}, h \rangle_\CH.
		\\ &= y +	\sum_{h \in \ST} \frac{\Upsilon_f[\Lambda^*(h)](y)}{\Vert h \Vert_{\CH}^2} \langle \mathbf{Z}, h \rangle_\CH.
		\\ &= y +	\sum_{h \in \ST} \frac{\bar{\Upsilon}_f[h](y)}{\Vert h \Vert_{\CH}^2} \langle \mathbf{Z}, h \rangle_\CH.
	\end{equs}
\end{proof}

\subsection{B-series solutions and controlled rough paths}
The following morphism property of elementary differentials ensures that the composition of Taylor coefficients is algebraically compatible with the composition of rough paths (Chen's relation), which further yields the composition of solutions and guarantees B-series as solutions.
\begin{lemma}\label{lemma:morphism_elementary_differential1}
	Elementary differentials are momorphisms with respect to the $\bstar_{\BCK}$ product, i.e., for any $\sigma_j, \tau \in \CT$
	\begin{equs}
	\Upsilon_f\left[ 	\left(\prod_{j=1}^{\odot n_\sigma}  \sigma_j \right) \bstar_{\BCK} \tau \right] (y)
	= 
	\sum_{b \in [m]^{n_{\sigma}} } 
	\left(
	\prod_{j=1}^{n_\sigma} 
	\Upsilon_f^{b_j}	[\sigma_j](y)
	\right)
	 \prod_{j=1}^{n_\sigma} \partial_{b_j} \Upsilon_f[\tau](y).
	\end{equs}
\end{lemma}
\begin{proof}
	We first prove the case when $n_\sigma = 1$ by induction. The proof of the case when $n_\sigma >1 $ follows from the universal property from  the Guin-Oudom construction \cite{Guin1,Guin2} that builds $ \bstar_{\BCK} $ out of the grafting product.
	Let $\tau = \mathcal{B}^{\alpha}_+	\left(\prod_{i=1}^{ \odot n_{\tau}}  \tau_i\right) $
	and suppose that, for any $\tau_i$,
	\begin{equs}
		 \Upsilon \left[ \sigma \bstar_{\BCK}  	  \tau_i \right](y)
		=
		\sum_{b \in [m]}
		\Upsilon_f^b	[\sigma](y)
		\partial_b \Upsilon_f[ \tau_i](y).
	\end{equs}
	 By the definition of elementary differentials,
	\begin{equs}
		\Upsilon_f\left[ \CB_+^{\alpha}	\left(\prod_{i=1}^{ \odot n_{\tau}}  \tau_i\right) \right] (y)
		= 
		\sum_{c \in [m]^{n_\tau}} \left(\prod_{i=1}^{n_{\tau}} \partial_{c_i}f_\alpha (y)\right)  \prod_{i=1}^{n_{\tau}}\Upsilon_f^{c_i}[ \tau_i](y).
	\end{equs}
	Then, the Leibniz rule implies
	\begin{equs}
		\sum_{b \in [m]}
		\Upsilon_f^b	[\sigma] \partial_b \Upsilon_f[\tau]
		&= 
		\sum_{b \in [m]}
		\Upsilon_f^b	[\sigma] (y)
		\partial_b \left(
	\sum_{c \in [m]^{n_\tau}} \left(\prod_{i=1}^{n_{\tau}} \partial_{c_i}f_\alpha (y)\right)  \prod_{i=1}^{n_{\tau}}\Upsilon_f^{c_i}[ \tau_i](y).
		\right)
		\\&=
		\sum_{b \in [m]}
		\sum_{c \in [m]^{n_\tau}}
		\Upsilon_f^b	[\sigma](y)
		 \left(\left(  \partial_b \prod_{i=1}^{n_{\tau}}  \partial_{c_i}f_\alpha (y)\right)  \prod_{i=1}^{n_{\tau}} \Upsilon_f^{c_i}[ \tau_i](y)\right)
		\\& +
		\sum_{b \in [m]}
		\sum_{c \in [m]^{n_\tau}}
		\left(\prod_{i=1}^{n_{\tau}} \partial_{c_i}f_\alpha (y)\right)  
		\Upsilon_f^b	[\sigma](y)
		\partial_b\left(\prod_{i=1}^{n_{\tau}}\Upsilon_f^{c_i}[ \tau_i](y)\right).
	\end{equs}
	Notice, from the definition of elementary differentials, that
	\begin{equs}
		&\sum_{b \in [m]}
		\sum_{\mathbf{c} \in [m]^{n_\tau}}
		\Upsilon_f^b	[\sigma](y)
		\left(\left(  \partial_b \prod_{i=1}^{n_{\tau}}  \partial_{c_i}f_\alpha (y)\right)  \prod_{i=1}^{n_{\tau}} \Upsilon_f^{c_i}[ \tau_i](y)\right)
		\\=&
		\Upsilon_f\left[ \CB_+^{\alpha}\left( \sigma \odot \prod_{i=1}^{\odot n_{\tau}}  \tau_i\right)\right].
	\end{equs}
	Moreover, by the induction hypothesis and by the Leibniz rule on $\partial$ and on $\bstar_{\BCK}$,
	\begin{equs}
	&	\sum_{b \in [m]}
		\sum_{c \in [m]^{n_\tau}}
		\left(\prod_{i=1}^{n_{\tau}} \partial_{c_i}f_\alpha (y)\right)  
		\Upsilon_f^b	[\sigma](y)
		\partial_b\left(\prod_{i=1}^{n_{\tau}}\Upsilon_f^{c_i}[ \tau_i](y)\right)
	\\ & 	= \Upsilon_f \left[ \mathcal{B}_+^{\alpha}\left( \sigma \bstar_{\BCK} 	\prod_{i=1}^{\odot n_{\tau} }  \tau_i\right) \right](y).
	\end{equs}
We conclude from the fact that
\begin{equs}
	\sigma \bstar_{\BCK} \tau = \mathcal{B}_+^{\alpha}\left( \sigma \odot	\prod_{i=1}^{\odot n_{\tau} }  \tau_i\right) + \mathcal{B}_+^{\alpha}\left( \sigma \bstar_{\BCK} 	\prod_{i=1}^{\odot n_{\tau} }  \tau_i\right). 
	\end{equs}
\end{proof}

Following from this morphism property of elementary differentials, B-series have the composition law. This ensures that the ansatz is preserved while composing local solutions to a ``flow", which is crucial in the solution theory and is at the heart of the concept of controlled rough paths introduced in \cite{Gubinelli2004,Gub06}.
\begin{proposition}\label{prop:composition}
	The Butcher-Connes-Kreimer $\CT$-B-series satisfies the following
	composition law.
	\begin{equs}
	  B_{\CT}(B_{\CT}(y,\mathbf{Z^1} ), \mathbf{Z^2})
	  &=
	  B_{\CT}(y,\mathbf{Z^1} \star_{\BCK} \mathbf{Z^2})
	\end{equs}
	for any $\mathbf{Z^1}, \mathbf{Z^2} \in \CH^*_{\BCK}$. 
%	Here, 
%	for $ \mathbf{Z} = \sum_{h \in \SH}\frac{\langle \mathbf{Z}, h \rangle}{S(h)}h^*$, the notation reads
%	\begin{equs}
%	\mathbf{Z}/ S(\cdot) = \sum_{h \in \SH}\frac{\langle \mathbf{Z}, h \rangle}{S(h)^2}h^*.
%	\end{equs}
\end{proposition}
\begin{proof}
For the simplicity in notations, we identify the spaces $\CH_{\BCK}$ and $\CH^*_{\BCK}$ since they are isomorphic.
	By the definition of the $\CT$-Butcher series,
		\begin{equs}		
		B_{\CT}(B_{\CT}(y,\mathbf{Z^1}), \mathbf{Z^2}) 
		=
		B_{\CT}(y,\mathbf{Z^1} ) 
		+ 
		\sum_{\tau \in \CT} \frac{\Upsilon_f[\tau](B_{\CT}(y,\mathbf{Z^1}))}{S(\tau)} 
		\langle \mathbf{Z^2} , \tau \rangle.
	\end{equs}
	The Taylor expansion of the elementary differentials around $y$ yields
	\begin{equs}
		&\Upsilon_f[\tau](B_{\CT}(y,\mathbf{Z^1} )) 
		\\=&
		\sum_{\beta \in \mathbb{N}^m}
		\frac{1}{\beta!}\partial_\beta\Upsilon_f[\tau](y)
		\prod_{i=1}^m
		\left(
		B_{\CT}^i(y,\mathbf{Z^1} )-y^i
		\right)^{\beta_i}
		\\=&
		\sum_{\beta \in \mathbb{N}^m}
		\frac{1}{\beta!}\partial_\beta\Upsilon_f[\tau](y)
		\prod_{i=1}^m
		\left(\sum_{\sigma_i \in \CT} \frac{\Upsilon^i_f[\sigma_i](y)}{S(\sigma_i)} \langle \mathbf{Z^1} , \sigma _i \rangle\right)^{\beta_i}
		\\=&
		\sum_{\beta \in \mathbb{N}^m}
		\frac{1}{\beta!} 	\partial_\beta\Upsilon_f[\tau](y)
		\sum_{\sigma_i^{k_i} \in \CT}
		\prod_{i=1}^m
		\prod_{k_i=1}^{\beta_i}
		\frac{\Upsilon^i_f[\sigma_i^{k_i}](y)}{S(\sigma_i^{k_i})} \langle \mathbf{Z^1}, \sigma_i^{k_i} \rangle
		\\ = & \sum_{\Cf = \prod_{i=1}^{\odot n} \sigma_i^{\odot r_i}} \frac{\prod_{i=1}^n S(\sigma_i)^{r_i}}{S(\Cf)} \sum_{b^i_j \in [m]^{r_i} } 
\prod_{i=1}^n 	\prod_{j=1}^{r_i}	\left(
		\prod_{i=1}^{n} 
		\frac{\Upsilon_f^{b_j}	[\sigma_i](y)}{S(\sigma_i)}
	\langle \mathbf{Z^1}, \sigma_i^{} \rangle	\right)
\\ & 	\prod_{i=1}^n 	\prod_{j=1}^{r_i}	  \partial_{b_j^i} \Upsilon_f[\tau](y)
	\end{equs}
	where $\beta$ is a multi-index with $m$ entries, the multi-index factorial $\beta! = \prod_{i=1}^m \beta_i!$, and $\partial_\beta$ is the the abbreviation of the differential operator $\prod_{i=1}^m\partial_{i}^{\beta_i}$. In the last identity, we assume that the $\sigma_i$ are pairwise disjoint trees and we have used the notation $ \sigma_i^{\odot r_i} = \prod_{j=1}^{\odot r_i} \sigma_i $. Using the fact that $\mathbf{Z}_1$ is a character combined with Lemma \ref{lemma:morphism_elementary_differential1}, one gets
	\begin{equation*}
		\Upsilon_f[\tau](B_{\CT}(y,\mathbf{Z^1} )) = B_{\CT}(y,\mathbf{Z^1} \star_{\BCK} \tau ).
	\end{equation*}
	We have also used the fact that $ \Upsilon_f $ is zero when evaluated on a forest with more than one tree.
	Therefore,
		\begin{equs}		
		B_{\!\ST}(B_{\!\ST}(y,\mathbf{Z^1}), \mathbf{Z^2}) 
	&	=
		B_{\CT}(y,\mathbf{Z^1} ) +
		\sum_{\tau \in \CT} \frac{B_{\CT}(y,\mathbf{Z^1} \star_{\BCK} \tau )}{S(\tau)} \langle \mathbf{Z^2}, \tau \rangle 
	\\ & = 		\sum_{\Cf \in \CF} \frac{B_{\CT}(y,\mathbf{Z^1} \star_{\BCK} \Cf )}{S(\Cf)} \langle \mathbf{Z^2}, \Cf \rangle 
	\\ & =  B_{\CT}(y,\mathbf{Z^1} \star_{\BCK} \mathbf{Z^2}).
	\end{equs}
\end{proof}
\begin{remark}
	The proof above follows the proof of \cite[Theorem 4.6]{BR23} which composes Regularity Structures B-series.
\end{remark}

\begin{corollary} \label{composition_H}
		The $\ST$-Butcher series satisfies the following
	composition law
	\begin{equs}
		B_{\ST}(B_{\ST}(y,\mathbf{Z^1} ), \mathbf{Z^2})
		&=
		B_{\ST}(y,\mathbf{Z^1} \star \mathbf{Z^2})
	\end{equs}
	for any $\mathbf{Z^1}, \mathbf{Z^2} \in \CH^*$. 
	\end{corollary}
	\begin{proof} 
		One has from Proposition \ref{symmetry_H}
		\begin{equs}
			B_{\ST}(y,\mathbf{Z^1} )  
%			&=  y +	\sum_{h \in \ST} \frac{\bar{\Upsilon}_f[h](y)}{\Vert h \Vert_\CH^2}\langle\mathbf{\bar{Z}^1}, h \rangle_{\CH}
%			\\ & =  y +	\sum_{\tau \in \CT} \frac{\Upsilon_f[\tau](y)}{S(\tau)}\langle \mathbf{\bar{Z}^1}, \Lambda( \tau) \rangle_{\CH}
%			\\ & 
			= 	B_{\CT}(y,\Lambda^*(\mathbf{Z^1}) ).
		\end{equs}
		Then, from Proposition \ref{prop:composition}, one has
		\begin{equs}
			B_{\CT}(B_{\CT}(y,\Lambda^*(\mathbf{Z^1}) ), \Lambda^*(\mathbf{Z^2}))
			&=
			B_{\CT}(y,\Lambda^*(\mathbf{Z^1}) \star_{\BCK} \Lambda^*( \mathbf{Z^2}))
			\end{equs}
			We conclude from the morphism property of $\Lambda^*$ for the product $\star$ that implies
			\begin{equs}
					B_{\CT}(y,\Lambda^*(\mathbf{Z^1}) \star_{\BCK} \Lambda^*( \mathbf{Z^2})) = B_{\CT}(y,\Lambda^*(\mathbf{Z^1} \star \mathbf{Z^2})) = B_{\ST}(y,\mathbf{Z^1} \star \mathbf{Z^2}).
			\end{equs}
		\end{proof}

It is shown in \cite[Theorem 5.2]{Gub06} that the coefficients of branched rough paths B-series form controlled rough paths. This enables the application of the Sewing Lemma \cite[Lemma 3.1]{HK15} to the solution ansatz, guaranteeing its convergence when composing local solutions to a flow, which was first constructed in \cite{Gubinelli2004} and adapted for branched rough paths in \cite{Gub06}. \cite[Proposition 3.8]{HK15} also shows the equivalence between the unique controlled rough path solution and the B-series solution.
Now we will show that the result of \cite[Theorem 5.2]{Gub06} is actually a consequence of  the composition law of B-series (Proposition \ref{prop:composition}). 
In fact, the proofs of \cite[Theorem 5.2]{Gub06} and \cite[Lemma 3.10]{HK15} correspond to composition of B-series. 

There are two equivalent definitions for controlled rough paths in the literature. The first one is \cite[Definition 8.1]{Gub06}, which is more in the B-series style, and the other one is a definition with 
a more Hopf algebraic flavour \cite[Definition 3.2]{HK15}. Here, we will stick to the latter one, since we will use the main result of  \cite{GLMZ25} which is based on it. 
 Let $\ST_{\le N} := \ST \cap (\bigoplus_{n \le N} \CH_n)$.
%and $\SF_{\le N} := \SF \cap (\bigoplus_{n \le N} \CH_{n})$
\begin{definition}[Controlled rough path]
	\label{control_rough_path}
	Let $\mathbf{X}$ be a $\gamma$-H\"{o}lder $\CH$-rough path, and let $N$ be the largest integer such that $N\gamma \le 1$. An $\mathbf{X}$-controlled rough path is a path $\mathbf{W} : [0, T] \rightarrow \CH$
	satisfying:
	\begin{itemize}
		\item For any $s \le t \in [0,T]$,
		\begin{equs}
			\langle \mathbf{1}, \mathbf{W}_t \rangle
			= W_t = W_s +\sum_{h \in \ST_{\le N-1}} \langle h, \mathbf{W}_s \rangle \langle \mathbf{X}_{s,t}, h \rangle+ R_{s,t}^{N}
		\end{equs}
		where $R_{s,t}^{N} \le C_1 |t-s|^{N\gamma}$ for some constant $C_1$.
		\item For any $h \in \ST_{\le N-1}$,
		\begin{equs}
			\langle h, \mathbf{W}_t \rangle
			=
			\sum_{g \in \ST_{\le N-1}} 
			\langle g, \mathbf{W}_s \rangle
			\langle  \mathbf{X}_{s,t} \star  h/\Vert h \Vert_{\CH}^2,  g \rangle 
			+R_{s,t}^\tau
			,
		\end{equs}
		where the remainder $R_{s,t}^h \le C_2 |t-s|^{(N-|h|)\gamma} $ for some constant $C_2$.
		\item For any $h \in \CH \setminus \ST$, 
		\begin{equs}
		\langle h, \mathbf{W} \rangle= 0.
		\end{equs}
	\end{itemize}
\end{definition}
\begin{remark}
Here we have a $\Vert h \Vert_{\CH}^2$ difference from \cite[Definition 3.2]{HK15} and \cite[Definition 5.4]{GLMZ25} due to the different definition of the inner product (see Remark \ref{remark:pairing} for details).
%In \cite[Definition 5.4]{GLMZ25}, the controlled rough path is defined for every $h \in \CH_{\BCK}$ but then later in \cite[Sec. 5.3]{GLMZ25} they restrict it to $\CF$. We therefore set up the definition in this case. 
\end{remark}

\begin{corollary}\label{corollary:controlled_RP}
	The elementary differentials of $\ST$-Butcher series form a controlled rough path $\mathbf{Y}: [0,T] \rightarrow \ST$ through
	\begin{equs}
		\langle h , \mathbf{Y}_t \rangle :=\frac{\bar{\Upsilon}_f[h](Y_t)}{\Vert h \Vert_{\CH}^2}
	\end{equs} 
	with $ h \in \ST^* $ and 
	\begin{equs}\label{eq:Y_t}
		\langle  \mathbf{1}, \mathbf{Y}_t \rangle
		:= Y_t =Y_s + \sum_{h \in \ST_{\le N-1}} \langle h, \mathbf{Y}_s \rangle \langle \mathbf{X}_{s,t}, h \rangle + R_{s,t}^{N}.
	\end{equs}
\end{corollary}
\begin{proof}
By Corollary \ref{composition_H}
\begin{equs}
	B_{\!\ST}(Y_t, \mathbf{Z}) =	B_{\!\ST}(B_{\!\ST}(Y_s, \mathbf{X}_{s,t}), \mathbf{Z})
	=
	B_{\!\ST}(Y_s, \mathbf{X}_{s,t} \star \mathbf{Z}),
\end{equs}
which implies that 
	\begin{equs}
&	\sum_{h \in \ST_{\le N-1}}	
\langle h, \mathbf{Y}_t\rangle 
\langle \mathbf{Z}, h \rangle 
		\\=& 
		\sum_{g \in \ST_{\le N-1}} \langle g, \mathbf{Y}_s \rangle \langle \mathbf{X}_{s,t} \star \mathbf{Z}, g \rangle 
		+R_{s,t}^{N}
%		\\&=
%		\sum_{\sigma \in \ST} \mathbf{Y}_s^\sigma \langle \mathbf{X}_{s,t} \otimes \mathbf{Z}, \Delta \sigma \rangle 
		\\=&
		\sum_{g \in \ST_{\le N-1}} 
		\sum_{k\in \ST_{\le N-|h|-1}}
		\sum_{h\in \ST_{\le N-1}}
		 \langle g, \mathbf{Y}_s \rangle 
		 \langle k \star h,  g \rangle  \frac{\langle \mathbf{X}_{s,t}, k \rangle}{\Vert k \Vert_{\CH}^2} \frac{\langle \mathbf{Z}, h \rangle}{\Vert h \Vert_{\CH}^2}+R_{s,t}^{N}.
	\end{equs}
Thus, let us match the coefficients in front of $\langle \mathbf{Z}, h \rangle$ and get
\begin{equs}
	\langle   h, \mathbf{Y}_t \rangle 
	&= \sum_{g \in \ST_{\le N-1}} 
	\sum_{k\in \ST_{\le N-|h|-1}}
	 \langle g, \mathbf{Y}_s \rangle 
	\langle k \star h,  g \rangle  
	\frac{\langle \mathbf{X}_{s,t}, k \rangle}{\Vert k \Vert_{\CH}^2\Vert h \Vert_{\CH}^2}+ R_{s,t}^{h}
%	\\&=
%	\sum_{\sigma \in \ST} \mathbf{Y}_s^\sigma \langle  \mathbf{X}_{s,t} \otimes  \tau , \Delta \sigma \rangle   \frac{1}{S(\tau)}
		\\&=
	\sum_{g \in \ST_{\le N-1}} 
	 \langle g, \mathbf{Y}_s \rangle 
	 \langle  \mathbf{X}_{s,t} \star  h/ \Vert h \Vert_{\CH}^2 ,  g \rangle + R_{s,t}^{h},
\end{equs}
where $R_{s,t}^h \le C_2 |t-s|^{(N-|h|)\gamma} $ for some constant $C_2$. Indeed, in the first line, we sum up $k \in \ST_{\le N-|h|-1}$, and thus the iterated integrals $\langle \mathbf{X}_{s,t}, k \rangle$ are $N-|h|-1$-H\"older continuous for all $k$. In fact, the first line 
 recovers the conditions in \cite[Definition 8.1]{Gub06} of controlled rough paths.
\end{proof}
\section{log-ODE solution theory of cocycle-RDEs}
\label{Sec::4}
In the log-ODE method, one needs the composition of elementary differentials with smooth functions. Therefore, we introduce the following definition. 
 \begin{definition}[elementary vector fields]
	For any $\tau \in \CT$ and any $y \in \mathbb{R}^m$, we define the elementary vector field	$\Upsilon_f[\tau] \{ \cdot \} (y)\colon  \CC^\infty(\mathbb{R}^m,\mathbb{R}^m) \to \CC^\infty(\mathbb{R}^m,\mathbb{R}^m)$ as the functional
	\begin{equs}\label{identification_1}
		\Upsilon_f[ \tau] \{\psi\}(y) := \sum_{b =1}^m \Upsilon_f^{b}[\tau](y)  \partial_b \psi(y).
	\end{equs}
	For generic element $\prod_{i=1}^{\odot n} \tau_i \in \CF$ we set
	\begin{equs}\label{identification_2}
		&\Upsilon_f[\mathbf{1} ] \{\psi\}(y) := \psi(y),
		\\&
		\Upsilon_f[\prod_{i=1}^{\odot n} \tau_i]  \{\psi\}(y) := \sum_{b_1,\ldots,b_n =1}^m \prod_{i=1}^n\Upsilon^{b_i}_f[ \tau_i](y)  \prod_{i=1}^{n}\partial_{b_i}\psi(y)\,.
	\end{equs}
	For $h  \in \CH$, we set
	\begin{equs}\label{identification_1_H}
		\bar{\Upsilon}_f[ h] \{\psi\}(y) := 	\Upsilon_f[ \Lambda^*(h)] \{\psi\}(y).
	\end{equs}
\end{definition}
One can immediately remark that the identity \eqref{identification_1} becomes  elementary differentials in Definition \eqref{def:elementary_differential} when $\psi = I_m y$, where $I_m$ is the $m$-dimensional identity matrix.

One can observe that the elementary differential has the following morphism property, which connects the expansions of solutions to the combinatorial objects. 
\begin{proposition} \label{morphism_element}
	The elementary differential is a homomorphism when one defines the product between elementary differentials as the composition, which means for any $u, v \in \CH^*_{\BCK}$ and $y \in \mathbb{R}$
	\begin{equs}
		\Upsilon_f[u \star_{\BCK} v]\{\psi\} (y)
		= 
		\Upsilon_f[u]\{\cdot\} (y)  \circ \Upsilon_f[v]\{\psi\} (y)
	\end{equs}
	where $\circ$ stands for the composition of elementary vector fields in the second input, i.e., $\Upsilon_f[u]\{\cdot\} (y)  \circ \Upsilon_f[v]\{\psi\} (y) = \Upsilon_f[u]  \left \{\Upsilon_f[v]\{\psi\}\right\}(y)$.
\end{proposition}

\begin{proof}
	Let us write $u = \prod_{j=1}^{ \odot n_u} \sigma_j$ and $v = \prod_{i=1}^{\odot n_v} \tau_i$ where $\sigma_j,\tau_i \in \CT$. By the definition of $\star$,
	\begin{equs}
		&\Upsilon_f[u \star_{\BCK} v]\{\psi\} (y)
		\\=& \Upsilon_f\left[(\Delta_\shuffle u) (\id \otimes \bstar_{\BCK} v)\right]\{\psi\} (y)
			\\=&
		\sum_{K \sqcup L = \{1,\ldots,n_u\}} \Upsilon_f[\mu\left(\prod_{k\in K}^{\odot}\sigma_k , 
		\prod_{l\in L}^{\odot}\sigma_l\bstar_{\BCK} v\right)]\{\psi\} (y)
		\\=&
		\sum_{K \sqcup L = \{1,\ldots,n_u\}}
		 \sum_{\{b_k\}_{k\in K} =1}^m  
		  \sum_{c_1,\ldots,c_{n_v} =1}^m 
		\sum_{L_1\sqcup\ldots \sqcup L_{n_v} =L}
		\\&
		\prod_{k \in K}\Upsilon^{b_k}_f[ \sigma_k](y)  
		 \prod_{i=1}^{n_v}
		\Upsilon_f^{c_i}[ \prod_{l \in L_i}^{\odot} \sigma_l \bstar_{\BCK}\tau_i](y)
		\prod_{i=1}^{n_v}\partial_{c_i}
		\prod_{k \in K}\partial_{b_k}\psi(y).
	\end{equs}
	On the other hand, by the definition of elementary vector fields and the Leibniz rule
	\begin{equs}
		 &\Upsilon_f[u]  \left \{\Upsilon_f[v]\{\psi\}\right\}(y) 
		\\ =&
		 \sum_{b_1,\ldots,b_{n_u} =1}^m \prod_{j=1}^{n_u}\Upsilon^{b_j}_f[ \sigma_j](y)  \prod_{j=1}^{n_u}\partial_{b_j}\Upsilon_f[v]\{\psi\}(y)
		 \\ =&
		  \sum_{b_1,\ldots,b_{n_u} =1}^m 
		   \sum_{c_1,\ldots,c_{n_v} =1}^m 
		  \sum_{K \sqcup L = \{1,\ldots,n_u\}}
		  \\&
		  \prod_{j=1}^{n_u}\Upsilon^{b_j}_f[ \sigma_j](y) 
		  \left(
		  \prod_{l \in L}\partial_{b_l}
		  \prod_{i=1}^{n_v}
		  \Upsilon_f^{c_i}[\tau_i](y)
		  \right)
		  \prod_{k \in K}\partial_{b_k}
		 \prod_{i=1}^{n_v}\partial_{c_i}
		 \psi(y).
	\end{equs}
	
	According to Lemma \ref{lemma:morphism_elementary_differential1}
	\begin{equs}
		& \sum_{\{b_l\}_{l\in L} =1}^m   
		 \prod_{l \in L}\Upsilon^{b_l}_f[ \sigma_l](y) 
		 \left(
		 \prod_{l \in L}\partial_{b_l}
		 \prod_{i=1}^{n_v}
		 \Upsilon_f^{c_i}[\tau_i](y)
		 \right) 
		 \\=&
		  \sum_{\{b_l\}_{l\in L} =1}^m   
		  \sum_{L_1\sqcup\ldots \sqcup L_{n_v} =L}
		  \prod_{i=1}^{n_v}
		   \left(
		 \prod_{l \in L_i}\Upsilon^{b_l}_f[ \sigma_l](y) 
		 \prod_{l \in L_i}\partial_{b_l}
		 \Upsilon_f^{c_i}[\tau_i](y)
		 \right) 
		 \\=&
		 \sum_{L_1\sqcup\ldots \sqcup L_{n_v} =L}
		 \prod_{i=1}^{n_v}
		 \Upsilon_f^{c_i}[ \prod_{l \in L_i}^{\odot} \sigma_l \bstar_{\BCK} \tau_i](y).
	\end{equs}
	Finally, we have
	\begin{equs}
	&\Upsilon_f[u]  \left \{\Upsilon_f[v]\{\psi\}\right\}(y) 
	\\ =&
	\sum_{c_1,\ldots,c_{n_v} =1}^m 
	\sum_{K \sqcup L = \{1,\ldots,n_u\}}
	 \sum_{\{b_k\}_{k\in K} =1}^m  
	  \sum_{L_1\sqcup\ldots \sqcup L_{n_v} =L}
	  \\&
	  \prod_{k \in K}\Upsilon^{b_k}_f[ \sigma_k](y) 
	 \prod_{i=1}^{n_v}
	 \left(
	 \Upsilon_f^{c_i}[ \prod_{l \in L_i}^{\odot} \sigma_l \bstar_{\BCK}\tau_i](y) 
	 \right) 
	 \left(
	 \prod_{i=1}^{n_v}\partial_{c_i}
	 \prod_{k \in K}\partial_{b_k}
	 \psi(y)
	 \right)
%	 	\\ =&
%	 \sum_{K \sqcup L = \{1,\ldots,n_u\}}
%	 \sum_{\{b_k\}_{k\in K} =1}^m  
%	 \prod_{k \in K}\Upsilon^{b_k}_f[ \sigma_k](y) 
%	 \Upsilon_f[ \prod_{l \in L_i}^{(\mu)} \sigma_l \bar{\star}v]
%	 \left\{
%	 \prod_{k \in K}\partial_{b_k}
%	 \psi
%	  \right\}
%	  (y) 
%	 	\\=&
%	 \sum_{K \sqcup L = \{1,\ldots,n_u\}} \Upsilon_f[\mu\left(\prod_{k\in K}^{(\mu)}\sigma_k , 
%	 \prod_{l\in L}^{(\mu)}\sigma_l\bar{\star} v\right)]\{\psi\} (y)
	 \\=&
	 \Upsilon_f[u \star_{\BCK} v]\{\psi\} (y).
	\end{equs}
%	where we used Lemma \ref{lemma:morphism_elementary_differential} and the definition of elementary vector fields.
\end{proof}

\begin{proposition}\label{prop:key_indentity_2}
		For any $u \in \CH^*_{\BCK}$, $y \in \mathbb{R}^m$ and $\phi, \psi \in \CC^\infty$,  we have 
	\begin{equs}\label{leibniz_id}
		\Upsilon_f[ \Delta_\shuffle u] \{\phi \otimes \psi\}(y) = \Upsilon_f[u] \{\phi \psi\}(y)
	\end{equs}
	where for any 
	$u,v \in \CH^*_{\BCK}$, $$\Upsilon_f[u \otimes v] \{\phi \otimes \psi\}(y) := \Upsilon_f[u ] \{\phi \}(y)\Upsilon_f[ v] \{ \psi\}(y).$$
\end{proposition}
\begin{proof}
	We only have to prove the identity for the primitive elements of $\Delta_{\shuffle}$ that are exactly $ \CT $  since the primitive elements are ``generators" and both $\Delta_{\shuffle}$ and elementary vector fields are linear in $\CF^*$. One can generate the proof from the primitive elements to the symmetric algebra by induction on the cardinal  of $u \in \CF^*$ that is the number of trees in the forest $u$.
	 On the left-hand side of the equality
	\begin{equs}
		\Upsilon_f[ \Delta_\shuffle \tau] \{\phi \otimes \psi\} (y)&= \Upsilon_f[  \tau \otimes \mathbf{1}] \{\phi \otimes \psi\}(y) + \Upsilon_f[  \mathbf{1} \otimes \tau] \{\phi \otimes \psi\}(y)
		\\ & = \psi(y) \sum_{b =1}^m \Upsilon_f^{b}[\tau](y)  \partial_b \phi(y)  + \phi(y) \sum_{b =1}^m \Upsilon_f^{b}[\tau](y)  \partial_b \psi(y)
		\\ & = \sum_{b =1}^m \Upsilon_f^{b}[\tau](y)  \partial_b (\phi(y)\psi(y)) 
		\\ & = \Upsilon_f[\tau] \{\phi \psi\}(y)
	\end{equs}
	since $\Upsilon_f[\mathbf{1}] (\phi) = \phi$ and we have applied the Leibniz rule.
\end{proof}
\begin{remark}
	The proof for Proposition \ref{morphism_element} and Proposition \ref{prop:key_indentity_2} for the case of planar trees are performed in \cite{KL23}. Since those for $\CH_\BCK$ are needed in this paper and to be self-contained, we have decided to provide their proofs. 
\end{remark}

\begin{proposition} \label{flow_morphism}
We recall the morphism of Hopf algebras given below
	$$\Lambda^*:(\CH^*,\star,\Delta_\mu)\to(\CH_\BCK^*,\star_{\BCK},\Delta_\shuffle).$$
	Then
	\begin{align*}
		\bar{\Upsilon}_f[u \star v]\{\phi\} &= \bar{\Upsilon}_f[u]  \circ \bar{\Upsilon}_f[v]\{\phi\},\\
		\bar{\Upsilon}_f[ \Delta_\mu u] \{\phi \otimes \psi\} & = \bar{\Upsilon}_f[u] \{\phi \psi\}.
	\end{align*}
\end{proposition}

\begin{proof} In this proof, we will use the fact that $  \Lambda^*$ is a Hopf algebra morphism between $ (\CH^*, \star, \Delta_{\mu}) $ and $ (\CH_{\BCK}^*, \mu, \Delta_{\shuffle}) $. This implies
	\begin{equs} \label{lambda_star}
		\Lambda^*(u \star v) = \Lambda^*(u) \star_{\BCK} \Lambda^*( v)
		\end{equs}
		\begin{equs} \label{lambda_delta}
			\left( \Lambda^* \otimes \Lambda^* \right) \Delta_{\mu} = \Delta_{\shuffle} \Lambda^*.
		\end{equs}
For the first identity, one has
	\begin{align*}
		\bar{\Upsilon}_f[u \star v]\{\phi\} &= \Upsilon_f[\Lambda^*(u \star v)]\{\phi\}\\
		&= \Upsilon_f[\Lambda^*(u) \star_{\BCK} \Lambda^*(v)]\{\phi\}\\
		&= \Upsilon_f[\Lambda^*(u)]  \circ \Upsilon_f[\Lambda^*(v)]\{\phi\}\\
		&= \bar{\Upsilon}_f[u]  \circ \bar{\Upsilon}_f[v]\{\phi\}.
	\end{align*}
	where from the first to the second line, we have used \eqref{lambda_star}. From the second to the third line, we have applied Proposition \ref{morphism_element}.
	For the second identity, one has
	\begin{equs}
			\bar{\Upsilon}_f[ \Delta_\mu u] \{\phi \otimes \psi\}(y)&  = 
			\Upsilon_f[\left( \Lambda^* \otimes \Lambda^* \right) \Delta_\mu u] \{\phi \otimes \psi\}(y)
			\\ &  = 
			\Upsilon_f[ \Delta_{\shuffle} \Lambda^*( u)] \{\phi \otimes \psi\}(y)
			\\ & = 
			 \Upsilon_f[\Lambda^*( u)] \{\phi \psi\}(y)
			 \\ & = \bar{\Upsilon}_f[u] \{\phi \psi\}(y)
	\end{equs}
	where from the first to the second line, we have used \eqref{lambda_delta}.
	From the second to the third line, we have applied Proposition \ref{prop:key_indentity_2}.
\end{proof}

	\begin{proof}[of Theorem\ref{main_thm_1}]
		Firstly, Theorem \ref{thm:univBCK} states that the cocycle condition \eqref{eq:one_cocycle} in Assumption \ref{assumption_fixed_point} ensures the existence of the morphism $\Lambda^{*} $ and the existence of its dual map $\Lambda$.		
		Recall that the space $\CH$ splits as $\CH=\CI\oplus\CK$, with $\CI=\Im(\Lambda)$, and $\CK=\Im(\Lambda)^\perp$ the orthogonal of $\Im(\Lambda)$.
		We denote by $\ST$ an orthogonal basis of $\CI$, we naturally identify $\CI$ and $\CI^*$ via this basis.
Proposition \ref{symmetry_H} shows that the solution of cocycle-RDE admits a $\ST$-B-series expression whose elementary differentials are defined via \eqref{first_def_elementary}.
		
		Therefore, it is left to check that elementary differentials of $\ST$-B-series satisfy the two key identities in Assumption \ref{assumption_fixed_point}. For $\CH_{\BCK}^*$, Proposition \ref{morphism_element} shows that
		\eqref{key_identities_1} holds while \eqref{key_identities_2} is proved in 
		Proposition \ref{prop:key_indentity_2}. We can pass those two identities to $\CH^*$ by using the fact that $\Lambda$ is a morphism of Hopf algebras (see Proposition \ref{flow_morphism}).
	\end{proof}
	\begin{remark}
		One observes that Assumption \ref{flow_condition} requires the existence of the morphism (for defining elementary differentials) but not necessarily the one-cocycle condition. Roughly speaking, the one-cocycle condition implies the existence of the morphism. 
%		The definition of elementary differentials . 
		That is why log-ODE method is more general than the fixed point argument in this sense.
		However, one has to keep in mind that the one-cocycle condition is sufficient for showing the fixed-point. It is not proved to be necessary.
	\end{remark}

\section{No algebraic fixed point for multi-indices}
\label{Sec::5}
Another combinatorial set coming from singular SPDEs has been used for defining rough paths. They come equipped with a Hopf algebra structure necessary for the Chen's relation.  Let us introduce this Hopf algebra $\CH_{\tiny{\mathbf{M}}} (\mathrm{Span}_{\mathbb{R}}(\CF\! \CM), \cdot, \Delta^{\tiny{\mathbf{M}}})$. We are given abstract variables $(z_k)_{k\in \mathbb{N}}$, a multi-index $z^{\beta}$ with $ \beta : \mathbb{N} \rightarrow \mathbb{N} $ having a finite support given by
\begin{equs}
	z^{\beta} : = \prod_{k \in \mathbb{N}} z_k^{\beta(k)}
\end{equs}
which is a monomial in the variables $z_k$. In practice, one has to work with more variables of the form $z_{(i,k)}$ with $i \in [d]$ but we prefer to look at simple multi-indices for having lighter notations as the proof is exactly the same. We are interested in specific multi-indices called populated multi-indices. They 
 satisfy the \emph{population} condition:
\begin{equation} \label{population_ODE}
	[\beta] :=	\sum_{k \in \mathbb{N}} (1 - k)\beta(k)  = 1.
\end{equation}
One has a surjective map from the rooted trees to the populated multi-indices that we denote by $\Phi$ and given by
\begin{equation*}
	\Phi( \CB_{+}(\tau_1,...,\tau_n) ) = z_{n} \prod_{i=1}^n \Phi(\tau_i).
\end{equation*}
Denote the set of populated multi-indices by $ \mathcal{M} $. Then, we consider the symmetric space over populated multi-indices. This vector space is formed of forests of multi-indices which are collections of multi-indices without any order among them. The forest product $\tilde{\prod}_{i=1}^m z^{\beta_i}$ (or $z^\beta \cdot z^\alpha$) is the juxtaposition of multi-indices. The identity element of the forest product is the empty multi-index $z^{\mathbf{0}}$ which is the multi-indice with $\beta(k) =0$ for every $k \in \mathbb{N}$. We denote this space by $ \mathcal{F} \! \mathcal{M} $.

Norms and symmetry factors are two important quantities of multi-indices. We would use the following notations and formulae for them.
\begin{itemize}
	\item The norm of a multi-index $z^\beta$ is the sum of each element of $\beta$
	\begin{equs}
		|z^\beta| = \sum_{k \in \mathbb{N}} \beta(k).
	\end{equs}
	\item The symmetry factor of a multi-indice counts the total number of permuting $k$ children of each node $z_k$
	\begin{equs} \label{symmetry_factor_1}
		S_{\text{\tiny M}}(z^{\beta}) = \prod_{k \in \mathbb{N}} (k!)^{\beta(k)}.
	\end{equs}
\end{itemize}
The above-mentioned quantities can be extended and applied to a forest of multi-indices. Then the norms are
\begin{equs}
	\left|\tilde{\prod}_{i=1}^m z^{\beta_i}\right| = \sum_{i=1}^m |z^{\beta_i}|.
\end{equs}
The symmetry factor for a forest of multi-indices $\tilde{z}^{\tilde{\beta}} =\tilde{\prod}_{i=1}^m \left(z^{\beta_i}\right)^{r_i}$	with distinct $\beta_i$ is
\begin{equs}  \label{symmetry_factor_2}
	S_{\text{\tiny M}}\left(\tilde{z}^{\tilde{\beta}}\right) = \prod_{i=1}^m r_i!\left(S_{\text{\tiny M}}(z^{\beta_i})\right)^{r_i}.
\end{equs} 
In the sequel, we will also use the following symmetry factor:
\begin{equs}
	S_{\tiny{\text{ext}}}( \tilde{\prod}_{i=1}^n z^{\beta_i}) = 	\frac{	S_{\text{\tiny M}}( \tilde{\prod}_{i=1}^n z^{\beta_i}) }{S_{\text{\tiny M}}( \prod_{i=1}^n z^{\beta_i})} = \prod_{i=1}^m r_i!
\end{equs}
For two forests of multi-indices $\tilde{z}^{\tilde{\alpha}}$ and  $\tilde{z}^{\tilde{\beta}}$, the pairing is defined via the inner product
\begin{equs} \label{inner_product}
	\langle \tilde{z}^{\tilde{\alpha}}, \tilde{z}^{\tilde{\beta}} \rangle := S_{\text{\tiny M}}(\tilde{z}^{\tilde{\alpha}})\delta^{\tilde{\alpha}}_{\tilde{\beta}},
\end{equs}
where $\delta^{\tilde{\alpha}}_{\tilde{\beta}} = 1$, if ${\tilde{\alpha}}={\tilde{\beta}}$, otherwise it is equal to $0$.

Equipped with these notations, we are able to provide an explicit formula for the coproduct denoted by $ \Delta^{\tiny{\mathbf{M}}} $ from \cite[Theorem 3.5]{BH24}:
	\begin{equs} \label{explicit_coproduct_BCK_2}
		\begin{aligned}
	\Delta^{\tiny{\mathbf{M}}} z^{\beta} &= 
	z^{\mathbf{0}} \otimes z^\beta + z^\beta \otimes z^{\mathbf{0}}
	\\&+
	\sum_{\substack{{\beta  = \beta_1 + \cdots + \beta_n + \hat{\beta}} \\ n \in \mathbb{N}^*}} \frac{1}{S_{\tiny{\text{ext}}}( \tilde{\prod}_{i=1}^n z^{\beta_i})}
	\tilde{\prod}_{i=1}^n z^{\beta_i}    \otimes  \bar{D}^{n} z^{\hat{\beta}},
	\end{aligned}		 
\end{equs}
where the $\beta_i$ satisfy the population condition and the image of the map $\bar{D}^{n}$ lives in $\mathrm{Span}_{\mathbb{R}}(\CM)$. The sum $\sum_{\beta  = \beta_1 + \cdots + \beta_n + \hat{\beta}}$ does not count the order among $\beta_i$, which means $\beta  = \beta_1 + \cdots + \beta_n + \hat{\beta}$ is a partition over $\beta$. The map $ \bar{D}$ is given by
	\begin{equs} \label{definition_adjoint}
	\bar{D} z^{\hat{\beta}} = \sum_{k \in \mathbb{N}^*} 	 k\frac{\hat{\beta}(k-1)+1}{ \hat{\beta}(k)} 	 z_{k-1} \partial_{z_k} z^{\hat{\beta}}.
\end{equs}
The formula was first introduced in \cite[Theorem 13]{GMZ26}.
We illustrate the previous definitions recalling an example given in \cite[Example 3.6]{BH24}
We only give the full computation for $z^{\hat{\beta}_1} = z_0z_1z_2$
\begin{equation*}
	\bar D z^{\hat{\beta}_1} = 1\frac{1+1}{1}z_0^2z_2 + 2 \frac{1+1}{1} z_0z_1^2 = 2z_0^2z_2 + 4z_0z_1^2.
\end{equation*}
Other terms are listed in the following tables.
\begin{table}
	\centering
	\begin{tabular}{||c c c c c||} 
		\hline
		Partition $(j)$ & $n$& $\tilde\prod_{i=1}^nz^{\beta_i}$ & $z^{\hat{\beta}_j}$& $\bar{D}^{n} z^{\hat{\beta}} $ \\ [0.5ex] 
		\hline\hline
		1 & 1& $z_0$  & $z_0z_1z_2$& $2z_0^2z_2+4z_0z_1^2$\\[1ex] 
		\hline
		2 & 1& $z_0z_1$  & $z_0z_2$& $2z_0z_1$\\ [1ex]
		\hline
		3 & 1& $z_0^2z_2$ & $z_1$& $z_0$\\ [1ex]
		\hline
		4 & 2& $z_0 \cdot z_0$   & $z_1z_2$& $6 z_0z_1$\\ [1ex]
		\hline
		5 & 2& $z_0\cdot z_0z_1$  & $z_2$& $2z_0$\\ [1ex] 
		\hline
	\end{tabular}
\end{table}Finally we have
\begin{equs}
	\Delta(z^\beta) &= z^{\mathbf{0}} \otimes  z_0^2z_1z_2 +  z_0^2z_1z_2 \otimes z^{\mathbf{0}} + 2z_0 \otimes z_0^2z_2 + 4z_0 \otimes z_0z_1^2
	\\&+ 2 z_0z_1 \otimes z_0z_1 + z_0^2z_2 \otimes z_0 
	+ 3z_0 \odot z_0 \otimes z_0z_1 + 2 z_0\odot z_0z_1 \otimes z_0.
\end{equs}

	Let $\bar{\Delta}^{\tiny{\mathbf{M}}}$ the reduced coproduct, for $x \in \mathcal{F}\!\mathcal{M} $, we have \begin{equation*}
		\bar{\Delta}^{\mathbf{M}}(x)=\Delta^{\mathbf{M}}(x)-z^{\mathbf{0}}\otimes x - x\otimes z^{\mathbf{0}}.
	\end{equation*}
Below, we present examples of computations with this reduced coproduct which are needed for the main theorem of the section.
We use the notation $  \bar{\Delta}^{\mathbf{M}}_{m,n} $ which means that we consider only the terms of the form
$  \tilde{\prod}_{i=1}^m \tilde{z}^{\tilde{\beta}_i} \otimes \tilde{\prod}_{j=1}^n \tilde{z}^{\tilde{\alpha}_j} $ with the $ \tilde{\beta}_i $ and $\tilde{\alpha}_j$ being non-zero, when one computes $ \bar{\Delta}^{\mathbf{M}}_{m,n} x $ with $ x \in \mathcal{F} \! \mathcal{M} $. 

 One also has the identity
\begin{equation*}
	\bar{\Delta}^{\mathbf{M}} = \sum_{m,n \in \mathbb{N}} \bar{\Delta}^{\mathbf{M}}_{m,n}.
\end{equation*}
In the computations below, we will use a short hand notation for forest of multi-indices. All multi-indices are written with the elements in decreasing order, for example: $z_1z_0$ and never $z_0z_1$.
This notation allows us to omit the $\cdot$.
Indeed, with this convention, $z_0z_1z_0$ can only be parsed as $z_0\cdot z_1z_0$.
Thus $z_0z_0z_0z_0$ is indeed $z_0\cdot z_0\cdot z_0\cdot z_0$.
Below, we write the computations for forests of multi-indices $\tilde{z}^{\tilde{\beta}}$ with $ | \tilde{z}^{\tilde{\beta}}  | \in \{ 1,2,3\} $ and we give the expression of each non-zero $  \bar{\Delta}^{\mathbf{M}}_{m,n} $ component.

\[\begin{array}{c|c}
	& \bar{\Delta}^{\mathbf{M}}_{1,1}\\
	\hline
	z_0z_0 & 2z_0\otimes z_0\\
	z_1z_0 & z_0\otimes z_0
\end{array}
\quad
\begin{array}{c|c|c}
	& \bar{\Delta}^{\mathbf{M}}_{1,2} & \bar{\Delta}^{\mathbf{M}}_{2,1}\\
	\hline
	z_0z_0z_0 & 3z_0\otimes z_0z_0 & 3z_0z_0\otimes z_0\\
	z_0z_1z_0 & z_0\otimes z_0z_0 + z_0\otimes z_1z_0 & z_0z_0\otimes z_0 + z_1z_0\otimes z_0\\
	z_1z_1z_0 & 2z_0\otimes z_1z_0 & z_1z_0\otimes z_0 \\
	z_2z_0z_0 & 2z_0\otimes z_1z_0 & z_0z_0\otimes z_0
\end{array}\]

\[\begin{array}{c|c|c|c}
	& \bar{\Delta}^{\mathbf{M}}_{1,3} & \bar{\Delta}^{\mathbf{M}}_{2,2}\\
	\hline
	z_0z_0z_0z_0 & 4z_0\otimes z_0z_0z_0 & 6z_0z_0\otimes z_0z_0 \\
	z_0z_0z_1z_0 & z_0\otimes z_0z_0z_0 + 2z_0\otimes z_0z_1z_0 & 2z_0z_0\otimes z_0z_0 + z_1z_0\otimes z_0z_0 \\
	z_0z_1z_1z_0 & 2z_0\otimes z_0z_1z_0 + z_0\otimes z_1z_1z_0 & 2z_0z_0\otimes z_1z_0 + 2z_1z_0\otimes z_0z_0 \\
	z_0z_2z_0z_0 & 2z_0\otimes z_0z_1z_0 + z_0\otimes z_2z_0z_0 & 2z_0z_0\otimes z_1z_0 + z_0z_0\otimes z_0z_0 \\
	z_1z_0z_1z_0 & 2z_0\otimes z_0z_1z_0 & z_0z_0\otimes z_0z_0 + 2z_1z_0\otimes z_1z_0\\
	z_1z_1z_1z_0 & 3z_0\otimes z_1z_1z_0 & 6z_1z_0\otimes z_1z_0 \\
	z_2z_1z_0z_0 & 4z_0\otimes z_1z_1z_0 + 2z_0\otimes z_2z_0z_0 & 3z_0z_0\otimes z_1z_0 + 2z_1z_0\otimes z_1z_0 \\
	z_3z_0z_0z_0 & 3z_0\otimes z_2z_0z_0 & 3 z_0z_0\otimes z_1z_0 
\end{array}\]

\[\begin{array}{c|c|c|c}
	&  \bar{\Delta}^{\mathbf{M}}_{3,1}\\
	\hline
	z_0z_0z_0z_0 & 4z_0z_0z_0\otimes z_0\\
	z_0z_0z_1z_0 &  z_0z_0z_0\otimes z_0 + 2z_0z_1z_0\otimes z_0\\
	z_0z_1z_1z_0 &  2z_0z_1z_0\otimes z_0 + z_1z_1z_0\otimes z_0\\
	z_0z_2z_0z_0 &  z_0z_0z_0\otimes z_0 + z_2z_0z_0\otimes z_0\\
	z_1z_0z_1z_0 &  2z_0z_1z_0\otimes z_0\\
	z_1z_1z_1z_0 & 3z_1z_1z_0\otimes z_0\\
	z_2z_1z_0z_0 &  2z_0z_1z_0\otimes z_0 + z_2z_0z_0\otimes z_0\\
	z_3z_0z_0z_0 &  z_0z_0z_0\otimes z_0
\end{array}\]

\begin{theorem} \label{thm_cocycle_multi}
	Let $L$ a cocycle of degree $1$ of the Hopf algebra of multi-indices, then $L(z_0)=0$. 
	Moreover the obstruction appear in degree $4$, which is the minimal degree for such an obstruction to appear.
	\end{theorem}

\begin{proof}
	The idea of the proof is to use the defining property of a cocycle to compute the value of $L$ for $z^{\mathbf{0}}, z_0, z_1z_0, z_1z_1z_0, z_2z_0z_0$, and to notice that it forces $L(z^{\mathbf{0}})=0$.
	We need to go up to degree $4$ because the multi-index of lowest degree representing several trees is $z_2z_1z_0z_0$, thus one cannot hope to find an obstruction in degree less than $4$.
	The proof now reduces to a tedious exercise of linear algebra involving a matrix $8\times 12$.
	
	We recall that a cocycle $L$ is a map such that
	$$\Delta^{\mathbf{M}}(L(x))=x_{(1)}\otimes L(x_{(2)}) + L(x)\otimes z^{\mathbf{0}}$$
	where we use the Sweedler notation $\Delta^{\mathbf{M}}(x)=x_{(1)}\otimes x_{(2)}$.

	Since $z_0$ is the only multi-index with one vertex, we have $\lambda\in\mathbb{R}$ such that:
	$$L(z^{\mathbf{0}})=6\lambda z_0.$$
	We have $\bar{\Delta}^{\mathbf{M}}(z_0)=0$, and $\bar{\Delta}^{\mathbf{M}}(L(z_0))=z_0\otimes L(z^{\mathbf{0}})$.
	Thus we have $\mu\in\mathbb{R}$ such that 
	$$L(z_0)=6\lambda z_1z_0 + 6\mu(z_0z_0 - 2z_1z_0).$$
	We have $\bar{\Delta}^{\mathbf{M}}(z_1z_0)=z_0\otimes z_0$, and $\bar{\Delta}^{\mathbf{M}}(L(z_1z_0))=z_0\otimes L(z_0) + z_1z_0\otimes L(z^{\mathbf{0}})$.
	Thus we have $\gamma\in \mathbb{R}$ such that
	$$L(z_1z_0)=3\lambda z_1z_1z_0 + 2\mu(z_0z_0z_0 - 3z_2z_0z_0) + \gamma(2z_0z_0z_0 - 6z_0z_1z_0 + 3z_1z_1z_0).$$
	Before beginning the computations in degree $4$, let us compute $\ker(\bar{\Delta}^{\mathbf{M}}_{2,2})\cap\ker(\bar{\Delta}^{\mathbf{M}}_{3,1})$.
	One may check that $\ker(\bar{\Delta}^{\mathbf{M}}_{3,1})=\mathrm{Span}(v_1,v_2,v_3,v_4)$ with
	\begin{align*}
		v_1 & = z_0z_0z_0z_0 - 4z_3z_0z_0z_0\\
		v_2 & = z_0z_0z_1z_0 - z_1z_0z_1z_0 - z_3z_0z_0z_0\\
		v_3 & = 3z_0z_1z_1z_0 - 3z_1z_0z_1z_0 - z_1z_1z_1z_0\\
		v_4 & = z_0z_2z_0z_0 + z_1z_0z_1z_0 - z_2z_1z_0z_0 - z_3z_0z_0z_0.
	\end{align*}
	We may check that $\ker(\bar{\Delta}^{\mathbf{M}}_{2,2})\cap\ker(\bar{\Delta}^{\mathbf{M}}_{3,1})=\mathrm{Span}(w_1,w_2)$ with
	\begin{align*}
		w_1 & = z_0z_0z_0z_0 - 3z_0z_2z_0z_0 - 3z_1z_0z_1z_0 + 3z_2z_1z_0z_0 - z_3z_0z_0z_0\\
		w_2 & = 3z_0z_0z_0z_0 - 12z_0z_0z_1z_0 + 6z_0z_1z_1z_0 + 6z_1z_0z_1z_0 - 2z_1z_1z_1z_0\\
		\bar{\Delta}^{\mathbf{M}}(w_1) & = 4z_0\otimes z_0z_0z_0 - 12z_0\otimes z_0z_1z_0 + 12z_0\otimes z_1z_1z_0 \\
		\bar{\Delta}^{\mathbf{M}}(w_2) & = 0.
	\end{align*}
	Let us compute $L(z_1z_1z_0)$.
	We set
	\begin{align*}
		\nu_\mu & = -z_0z_0z_0z_0 + 6z_0z_0z_1z_0 - 6z_1z_0z_1z_0 - 2z_3z_0z_0z_0\\
		\bar{\Delta}^{\mathbf{M}}(\nu_\mu) & = 2z_0\otimes z_0z_0z_0 - 6z_0\otimes z_2z_0z_0 + 6z_1z_0\otimes z_0z_0 - 12z_1z_0\otimes z_1z_0.
	\end{align*}
	Since
	$$\bar{\Delta}^{\mathbf{M}}(L(z_1z_1z_0))= 2z_0\otimes L(z_1z_0) + 2z_1z_0\otimes L(z_0) + z_1z_1z_0\otimes L(z^{\mathbf{0}})$$
	we have
	\begin{multline*}
		\bar{\Delta}^{\mathbf{M}}(L(z_1z_1z_0) - \lambda z_1z_1z_1z_0 - 2\mu\nu_\mu) \\
		= \gamma(2z_0\otimes z_0z_0z_0 - 6z_0\otimes z_0z_1z_0 + 3z_0\otimes z_1z_1z_0). 
	\end{multline*}
	Thus, $L(z_1z_1z_0) - \lambda z_1z_1z_1z_0 - 2\mu\nu_\mu\in\ker(\bar{\Delta}^{\mathbf{M}}_{2,2})\cap\ker(\bar{\Delta}^{\mathbf{M}}_{3,1})$, and $\gamma=0$. 
	Let us compute $L(z_2z_0z_0)$.
	We set
	\begin{align*}
		\nu_\lambda & = -z_0z_0z_0z_0 + 6z_0z_2z_0z_0 - 2z_3z_0z_0z_0\\
		\bar{\Delta}^{\mathbf{M}}(\nu_\lambda) & = -4z_0\otimes z_0z_0z_0 + 12z_0\otimes z_0z_1z_0 + 6z_0z_0\otimes z_1z_0 + 6z_2z_0z_0\otimes z_0.
	\end{align*}
	Since
	$$\bar{\Delta}^{\mathbf{M}}(L(z_2z_0z_0))= 2z_0\otimes L(z_1z_0) + z_0z_0\otimes L(z_0) + z_2z_0z_0\otimes L(1)$$
	we have
	\begin{multline*}
		\bar{\Delta}^{\mathbf{M}}(L(z_1z_1z_0) - \lambda\nu_\lambda - \mu(z_0z_0z_0z_0 - 4z_3z_0z_0z_0)) \\
		= \lambda(-4z_0\otimes z_0z_0z_0 + 12z_0\otimes z_0z_1z_0 - 6z_0\otimes z_2z_0z_0). 
	\end{multline*}
	Thus $L(z_1z_1z_0) - \lambda\nu_\lambda - \mu(z_0z_0z_0z_0 - 4z_3z_0z_0z_0)\in\ker(\bar{\Delta}^{\mathbf{M}}_{2,2})\cap\ker(\bar{\Delta}^{\mathbf{M}}_{3,1})$, and $\lambda=0$ which implies $ L(z^{\mathbf{0}}) = 0 $.
\end{proof}

\begin{proposition} \label{prop_flow_multi} The muti-indices Hopf algebra satisfies Assumption \ref{flow_condition}. 
	\end{proposition}
	\begin{proof} The map $  \Phi$ is a morphism from $ (\CH_{\BCK}^*, \star) $ into $ (\CH_{\mathbf{M}}^*, \star_{\mathbf{M}}) $ where $ \star_{\mathbf{M}} $ is the graded dual of $\Delta_{\mathbf{M}}$. This follows from the morphism property given in \cite[Section 6.5]{LOT}.
		Then, the elementary differential $ \bar{\Upsilon}_f $ is given by
		\begin{equs} \label{explicit_elementary}
			\bar{\Upsilon}_f[z^{\beta}](y)  = \prod_{k \in \mathbb{N} } (f^{(k)}(y))^{\beta(k)}.
		\end{equs}
		where the previous definition is used for defining multi-indices B-series in \cite{BHE24}.
		From \cite[Section 6.5]{LOT}, one has for $ \tau \in \CT $
		\begin{equs} \label{morphism_elementary}
			\bar{\Upsilon}_f[\Phi(\tau)] = \Upsilon_f[\tau].
		\end{equs}
		which corresponds to \eqref{second_def_elementary}.
		The fact that this elementary differential satisfies \eqref{key_identities_1} and \eqref{key_identities_2} has been proved in \cite[Proposition 2.7]{BBH26} and \cite[Proposition 2.8]{BBH26} with the use of \eqref{explicit_elementary}.
		\end{proof}

\end{document}